\documentclass[12pt]{article}
\usepackage{graphicx} 
\usepackage{pgfplots}
\usepackage{tikz}
\usepackage{amsthm}
\usepackage{csquotes}
\usepackage{amsmath}
\usepackage{natbib}
\usepackage{graphicx}
\usepackage{setspace}
\definecolor{brightorange}{RGB}{255, 153, 51}
\definecolor{vividpurple}{RGB}{106, 15, 142}
\definecolor{violet}{RGB}{183,182,231}
\definecolor{tan}{RGB}{210,180,140}
\usepackage{amsfonts}
\newcounter{counter}
\numberwithin{counter}{section}
\newtheorem{theorem}[counter]{Theorem}
\newtheorem{proposition}[counter]{Proposition}
\newtheorem{lemma}[counter]{Lemma}
\newtheorem{corollary}[counter]{Corollary}
\newtheorem{definition}[counter]{Definition}
\usepackage{color}
\numberwithin{equation}{section}
\pgfplotsset{width = 10cm, compat=1.9}
\usepackage{longtable}
\usepackage{array}
\usepackage{lipsum}
\usepackage{mathrsfs}
\usepackage{float}
\usepackage[margin=1in]{geometry}
\usepackage{enumitem}
\usepackage{fancyhdr}
\usepackage{subcaption}
\usepackage{amssymb}
\usepackage{verbatim}

\title{A Graphical Approach to the Frobenius Number Problem in Three Variables}
\author{Xinxin Fang\\ \small The Episcopal Academy}
\begin{document}
\begin{titlepage}
    \centering
    \vspace*{2cm} 
    {\Huge\bfseries A Graphical Approach to the Frobenius Number Problem in Three Variables\par}
    \vspace{2cm}
    {\Large Xinxin Fang\par}
    \vspace{2cm}
\end{titlepage}
\onehalfspacing
\pagestyle{empty}
\begin{abstract}
The Frobenius number for a set of relatively prime positive integers, where the smallest integer in the set is at least 2, is the largest integer that cannot be expressed as a nonnegative linear combination of those integers. We analyze the Frobenius number for three variables by analyzing the lattice points associated with the line $ax+by = ab$ and its downward parallel translations, along with each lattice point's corresponding value $ax+by$. As an application of our graphical approach, we recover key theorems such as the results from \cite{sylvester1882} and \cite{selmer1977} by reducing them to a graphical problem of locating a lattice point, followed by an evaluation of the linear form $ax+by$ there. Therefore, an arithmetic problem of finding the Frobenius number is reduced to a visual/graphical problem of locating the lattice point, followed by an evaluation, both of which are simple tasks. We then derive relatively simple formulae for the general Frobenius number of three variables. Compared to known results, our results (Theorem \ref{thm:Final}) are not algorithmic in nature nor intricate, which constitutes a significant improvement. \par \vspace{12pt}
\noindent{\bf Keywords:} Frobenius number, graphical approach, lattice points
\end{abstract}
\newpage
\setcounter{page}{1}
\pagestyle{fancy}
\section{Background}
Given relatively prime positive integers $a_1, a_2, \ldots, a_n$ with $a_i \geq 2$, the Frobenius number for $a_1, a_2, \ldots, a_n$, $g(a_1, a_2, \ldots, a_n)$, is defined as the largest integer not representable as a nonnegative integer linear combination of the set $a_1, a_2, \ldots, a_n$. Any integer greater than the Frobenius number can be represented as a nonnegative integer linear combination of $a_1, a_2, \ldots, a_n$. \par Throughout this article, we assume that $a_i < a_{i+1}$. The result for the Frobenius number of two numbers is well-known, but currently, there is no explicit solution to the Frobenius problem for $n>3$, though many have developed algorithms to find the Frobenius number of $n=3$. (\cite{davison1994}, \cite{kennedy1992}, \cite{nijenhuis1972}, \cite{Rödseth1978}, \cite{tripathi2017}, \cite{wilf1978}, and \cite{selmerbeyer1978}). 
\section{Preliminary Results}
In addition to assuming that $a_i < a_{i+1}$, our final result (Theorem \ref{thm:Final}) assumes that $a$ and $b$ are relatively prime; however, the case where $a$ and $b$ are not relatively prime can be deduced by \cite{Johnson1960} result:
\begin{theorem}
Let {\rm{gcd}$(a,b) = d$} and $a,b,c$ are relatively prime, then
\begin{align*}
    g(a,b,c) = d*g(\frac{a}{d},\frac{b}{d},c)+c(d-1)
\end{align*}
\end{theorem}
This result allows us to assume that $a$ and $b$ are relatively prime, though the resulting $\frac{a}{d},\frac{b}{d},c$ do not need to be pairwise relatively prime.
\begin{lemma}\label{lem:+diophantine}
There exists a unique integer solution $(x_1, y_1)$ to the linear Diophantine equation $ax+by=c$ such that $0 \leq x_1 < b$, given that $a,b,c \in \mathbb{Z}^+$ and \rm{gcd}$(a,b) = 1$.
\end{lemma}
\begin{proof}
By Bezout's identity, there exists integers $(x_0,y_0)$ such that $ax_{0}+by_{0}=c$, and all the solutions to the linear Diophantine equation are given by $(x_{0}+kb,y_{0}-ka)$ for $k \in \mathbb{Z}$. To achieve a unique integer $k$ such that $0 \leq x_{0}+kb < b$, Euclid's division lemma gives a unique integer $k$ such that $x_{0}+kb=x_1$ for $0 \leq x_1 < b$.
\end{proof}
We define the region $D = \{(x,y)\in \mathbb{Z}^2 \mid 0 \leq x \leq (b-1), -(a-1)\leq y \leq (a-1), 0 < ax+by < ab\}$ for relatively prime integers $a$ and $b$. For a fixed pair of relatively prime integers $a$ and $b$ and for each point $(x,y)$ within $D$, we denote the corresponding linear form 
\begin{equation}\label{eq:1}
    f_{[a \ b]}({x,y}) = ax+by 
\end{equation}
We emphasize the connection between the lattice point $(x,y) \in \mathbb{Z}^2$ and the value $f_{[a \ b]}({x,y})$ at that point. For example, in $\mathbb{Z}^2$ with $[a \ b]= [7 \ 9]$, the corresponding value for the point $(4,1)$ is $f_{[7 \ 9]}({4,1}) = 7\cdot4+9\cdot1 = 37$ as shown in Figure \ref{fig:GraphSevenNine}. 
\par In terms of the graphical set up as seen in Figure \ref{fig:GraphSevenNine}, a lattice point $(x,y)$ is labeled green if their corresponding value $f_{[a \ b]}({x,y})$ is divisible by $a$ or $b$. We define the set point of points as $D^0 = \{(x,y) \in \mathbb{Z}^2 \mid x$ or $y = 0, 0 \leq ax+by \leq ab\}$. In general,  $\rvert D^0 \lvert = a+b+1$.  If the point $(x,y)$ is not green, it is either blue or red, depending on if $y$ is positive or negative respectively. That is, the blue points are the lattice point set $D^+=\{(x,y)\in D \mid x,y > 0\}$, and the red points are the lattice point set $D^-=\{(x,y)\in D \mid x>0,y<0\}$. Lemma \ref{lem:+diophantine} allows us to limit our observations to the two triangles $(D^+ \cup D^0)$ and $D^-$. We establish a one-to-one correspondence between blue points and red points. 
\begin{proposition}\label{prop:equal_blue_red}
For every blue (respectively red) point $(x,y)$, it has a corresponding red (respectively blue) point $(b-x,-y)$ such that $f_{[a \ b]}(x,y)+f_{[a \ b]}(b-x,-y) = ab$. 
\end{proposition} 
\begin{proof}
\begin{align*}
    f_{[a \ b]}(x,y)+f_{[a \ b]}(b-x,-y) &= ax+by + a(b-x) + b(-y) \\
    &= ax+by+ab-ax-by \\
    &= ab
\end{align*}
Alternatively, if a point $(x,y)$ with a corresponding value $f_{[a \ b]}(x,y)$ is a blue point (respectively red), it has a corresponding red (respectively blue) point $(b-x,-y)$ with a value of $f_{[a \ b]}(x,y)$ which has a corresponding value $ab-f_{[a \ b]}(x,y)$. Therefore, there are just as many blue points as red points or $\lvert D^+ \rvert = \lvert D^- \rvert$. 
Note that all the blue points have a corresponding value $f_{[a \ b]}(x,y)$ that is a linear combination of $a$ and $b$ with positive integer coefficients. All red points have a $f_{[a \ b]}(x,y)$ value that is in the form $ab$ minus a linear combination of $a$ and $b$ with positive integer coefficients.
\end{proof}
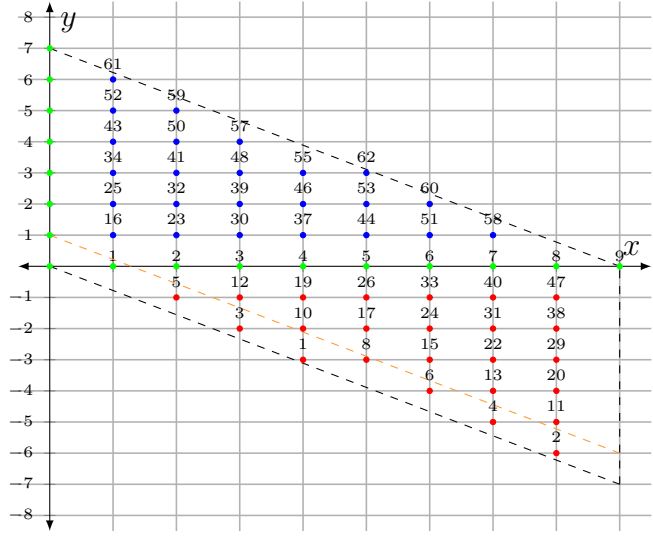
\begin{figure}
\begin{center}
\begin{tikzpicture}
  \begin{axis}[
    axis lines=center,
    xlabel=$x$,
    ylabel=$y$,
    xmin=0,
    xmax=9,
    ymin=-8,
    ymax=8,
    xticklabel style={above},
    xtick={0,1, ...,9},
    ytick={-8,-7, ...,7,8},
    grid=both,
    grid style={line width=0.2pt, draw=gray!30},
    major grid style={line width=0.6pt,draw=gray!60},
    ticklabel style={font=\tiny},
    enlargelimits={abs=0.5},
    axis line style={latex-latex},
  ]

  \addplot[
    domain=0:9,
    samples=100,
    color=black,
    dashed,
  ]{((63 - 7*x)/9)};

  \addplot[
    domain=0:9,
    samples=100,
    color=black,
    dashed,
  ]{((-7*x)/9)};
  \addplot[
    domain=0:9,
    samples=100,
    color=brightorange,
    dashed,
  ]{((-7*x)/9+1)};
  \addplot[
    domain=-10:10,
    samples=2,
    color=black,
    dashed
  ] coordinates {(9, -7) (9, 0)};

  \addplot[
    only marks,
    mark=*,
    mark options={color=red, fill=red, scale=.5},
  ] coordinates {(8, -1)};
  \node[above, font=\tiny] at (axis cs:8,-1) {47};

  \addplot[
    only marks,
    mark=*,
    mark options={color=red, fill=red, scale=.5},
  ] coordinates {(7, -1)};
  \node[above, font=\tiny] at (axis cs:7,-1) {40};
    \addplot[
    only marks,
    mark=*,
    mark options={color=red, fill=red, scale=.5},
  ] coordinates {(6, -1)};
  \node[above, font=\tiny] at (axis cs:6,-1) {33};
    \addplot[
    only marks,
    mark=*,
    mark options={color=red, fill=red, scale=.5},
  ] coordinates {(5, -1)};
  \node[above, font=\tiny] at (axis cs:5,-1) {26};
    \addplot[
    only marks,
    mark=*,
    mark options={color=red, fill=red, scale=.5},
  ] coordinates {(4, -1)};
  \node[above, font=\tiny] at (axis cs:4,-1) {19};
    \addplot[
    only marks,
    mark=*,
    mark options={color=red, fill=red, scale=.5},
  ] coordinates {(3, -1)};
  \node[above, font=\tiny] at (axis cs:3,-1) {12};
  \addplot[
    only marks,
    mark=*,
    mark options={color=red, fill=red, scale=.5},
  ] coordinates {(2, -1)};
  \node[above, font=\tiny] at (axis cs:2,-1) {5};
  \addplot[
    only marks,
    mark=*,
    mark options={color=red, fill=red, scale=.5},
  ] coordinates {(8, -2)};
  \node[above, font=\tiny] at (axis cs:8,-2) {38};
  \addplot[
    only marks,
    mark=*,
    mark options={color=red, fill=red, scale=.5},
  ] coordinates {(7, -2)};
  \node[above, font=\tiny] at (axis cs:7,-2) {31};
  \addplot[
    only marks,
    mark=*,
    mark options={color=red, fill=red, scale=.5},
  ] coordinates {(6, -2)};
  \node[above, font=\tiny] at (axis cs:6,-2) {24};
  \addplot[
    only marks,
    mark=*,
    mark options={color=red, fill=red, scale=.5},
  ] coordinates {(5, -2)};
  \node[above, font=\tiny] at (axis cs:5,-2) {17};
  \addplot[
    only marks,
    mark=*,
    mark options={color=red, fill=red, scale=.5},
  ] coordinates {(4, -2)};
  \node[above, font=\tiny] at (axis cs:4,-2) {10};
  \addplot[
    only marks,
    mark=*,
    mark options={color=red, fill=red, scale=.5},
  ] coordinates {(3, -2)};
  \node[above, font=\tiny] at (axis cs:3,-2) {3};
  \addplot[
    only marks,
    mark=*,
    mark options={color=red, fill=red, scale=.5},
  ] coordinates {(8, -3)};
  \node[above, font=\tiny] at (axis cs:8,-3) {29};
  \addplot[
    only marks,
    mark=*,
    mark options={color=red, fill=red, scale=.5},
  ] coordinates {(7, -3)};
  \node[above, font=\tiny] at (axis cs:7,-3) {22};
   \addplot[
    only marks,
    mark=*,
    mark options={color=red, fill=red, scale=.5},
  ] coordinates {(6, -3)};
  \node[above, font=\tiny] at (axis cs:6,-3) {15};
   \addplot[
    only marks,
    mark=*,
    mark options={color=red, fill=red, scale=.5},
  ] coordinates {(5, -3)};
  \node[above, font=\tiny] at (axis cs:5,-3) {8};
  \addplot[
    only marks,
    mark=*,
    mark options={color=red, fill=red, scale=.5},
  ] coordinates {(4, -3)};
  \node[above, font=\tiny] at (axis cs:4,-3) {1};
   \addplot[
    only marks,
    mark=*,
    mark options={color=red, fill=red, scale=.5},
  ] coordinates {(8, -4)};
  \node[above, font=\tiny] at (axis cs:8,-4) {20};
   \addplot[
    only marks,
    mark=*,
    mark options={color=red, fill=red, scale=.5},
  ] coordinates {(7, -4)};
  \node[above, font=\tiny] at (axis cs:7,-4) {13};
  \addplot[
    only marks,
    mark=*,
    mark options={color=red, fill=red, scale=.5},
  ] coordinates {(6, -4)};
  \node[above, font=\tiny] at (axis cs:6,-4) {6};
  \addplot[
    only marks,
    mark=*,
    mark options={color=red, fill=red, scale=.5},
  ] coordinates {(7, -5)};
  \node[above, font=\tiny] at (axis cs:7,-5) {4};
  \addplot[
    only marks,
    mark=*,
    mark options={color=red, fill=red, scale=.5},
  ] coordinates {(8, -5)};
  \node[above, font=\tiny] at (axis cs:8,-5) {11};
   \addplot[
    only marks,
    mark=*,
    mark options={color=red, fill=red, scale=.5},
  ] coordinates {(8, -6)};
  \node[above, font=\tiny] at (axis cs:8,-6) {2};

  \addplot[
    only marks,
    mark=*,
    mark options={color=blue, fill=blue, scale=.5},
  ] coordinates {(1, 1)};
  \node[above, font=\tiny] at (axis cs:1,1) {16};
  \addplot[
    only marks,
    mark=*,
    mark options={color=blue, fill=blue, scale=.5},
  ] coordinates {(2, 1)};
  \node[above, font=\tiny] at (axis cs:2,1) {23};
  \addplot[
    only marks,
    mark=*,
    mark options={color=blue, fill=blue, scale=.5},
  ] coordinates {(3, 1)};
  \node[above, font=\tiny] at (axis cs:3,1) {30};
  \addplot[
    only marks,
    mark=*,
    mark options={color=blue, fill=blue, scale=.5},
  ] coordinates {(4, 1)};
  \node[above, font=\tiny] at (axis cs:4,1) {37};
  \addplot[
    only marks,
    mark=*,
    mark options={color=blue, fill=blue, scale=.5},
  ] coordinates {(5, 1)};
  \node[above, font=\tiny] at (axis cs:5,1) {44};
\addplot[
    only marks,
    mark=*,
    mark options={color=blue, fill=blue, scale=.5},
  ] coordinates {(6, 1)};
  \node[above, font=\tiny] at (axis cs:6,1) {51};
  \addplot[
    only marks,
    mark=*,
    mark options={color=blue, fill=blue, scale=.5},
  ] coordinates {(7, 1)};
  \node[above, font=\tiny] at (axis cs:7,1) {58};
  \addplot[
    only marks,
    mark=*,
    mark options={color=blue, fill=blue, scale=.5},
  ] coordinates {(1, 2)};
  \node[above, font=\tiny] at (axis cs:1,2) {25};
  \addplot[
    only marks,
    mark=*,
    mark options={color=blue, fill=blue, scale=.5},
  ] coordinates {(2, 2)};
  \node[above, font=\tiny] at (axis cs:2,2) {32};
  \addplot[
    only marks,
    mark=*,
    mark options={color=blue, fill=blue, scale=.5},
  ] coordinates {(3, 2)};
  \node[above, font=\tiny] at (axis cs:3,2) {39};
  \addplot[
    only marks,
    mark=*,
    mark options={color=blue, fill=blue, scale=.5},
  ] coordinates {(4, 2)};
  \node[above, font=\tiny] at (axis cs:4,2) {46};
  \addplot[
    only marks,
    mark=*,
    mark options={color=blue, fill=blue, scale=.5},
  ] coordinates {(5, 2)};
  \node[above, font=\tiny] at (axis cs:5,2) {53};
  \addplot[
    only marks,
    mark=*,
    mark options={color=blue, fill=blue, scale=.5},
  ] coordinates {(6, 2)};
  \node[above, font=\tiny] at (axis cs:6,2) {60};
   \addplot[
    only marks,
    mark=*,
    mark options={color=blue, fill=blue, scale=.5},
  ] coordinates {(1, 3)};
  \node[above, font=\tiny] at (axis cs:1,3) {34};
   \addplot[
    only marks,
    mark=*,
    mark options={color=blue, fill=blue, scale=.5},
  ] coordinates {(2, 3)};
  \node[above, font=\tiny] at (axis cs:2,3) {41};
  \addplot[
    only marks,
    mark=*,
    mark options={color=blue, fill=blue, scale=.5},
  ] coordinates {(3, 3)};
  \node[above, font=\tiny] at (axis cs:3,3) {48};
  \addplot[
    only marks,
    mark=*,
    mark options={color=blue, fill=blue, scale=.5},
  ] coordinates {(4, 3)};
  \node[above, font=\tiny] at (axis cs:4,3) {55};
  \addplot[
    only marks,
    mark=*,
    mark options={color=blue, fill=blue, scale=.5},
  ] coordinates {(5, 3)};
  \node[above, font=\tiny] at (axis cs:5,3) {62};
  \addplot[
    only marks,
    mark=*,
    mark options={color=blue, fill=blue, scale=.5},
  ] coordinates {(1, 4)};
  \node[above, font=\tiny] at (axis cs:1,4) {43};
   \addplot[
    only marks,
    mark=*,
    mark options={color=blue, fill=blue, scale=.5},
  ] coordinates {(2, 4)};
  \node[above, font=\tiny] at (axis cs:2,4) {50};
   \addplot[
    only marks,
    mark=*,
    mark options={color=blue, fill=blue, scale=.5},
  ] coordinates {(3, 4)};
  \node[above, font=\tiny] at (axis cs:3,4) {57};
  \addplot[
    only marks,
    mark=*,
    mark options={color=blue, fill=blue, scale=.5},
  ] coordinates {(1, 5)};
  \node[above, font=\tiny] at (axis cs:1,5) {52};
  \addplot[
    only marks,
    mark=*,
    mark options={color=blue, fill=blue, scale=.5},
  ] coordinates {(2, 5)};
  \node[above, font=\tiny] at (axis cs:2,5) {59};
   \addplot[
    only marks,
    mark=*,
    mark options={color=blue, fill=blue, scale=.5},
  ] coordinates {(1, 6)};
  \node[above, font=\tiny] at (axis cs:1,6) {61};
  \addplot[
    only marks,
    mark=*,
    mark options={color=green, fill=green, scale=.5},
  ] coordinates {(0, 1)};
  \addplot[
    only marks,
    mark=*,
    mark options={color=green, fill=green, scale=.5},
  ] coordinates {(0, 2)};
  \addplot[
    only marks,
    mark=*,
    mark options={color=green, fill=green, scale=.5},
  ] coordinates {(0, 3)};
  \addplot[
    only marks,
    mark=*,
    mark options={color=green, fill=green, scale=.5},
  ] coordinates {(0, 4)};
  \addplot[
    only marks,
    mark=*,
    mark options={color=green, fill=green, scale=.5},
  ] coordinates {(0,5)};
  \addplot[
    only marks,
    mark=*,
    mark options={color=green, fill=green, scale=.5},
  ] coordinates {(0, 6)};
  \addplot[
    only marks,
    mark=*,
    mark options={color=green, fill=green, scale=.5},
  ] coordinates {(1, 0)};
  \addplot[
    only marks,
    mark=*,
    mark options={color=green, fill=green, scale=.5},
  ] coordinates {(2, 0)};
  \addplot[
    only marks,
    mark=*,
    mark options={color=green, fill=green, scale=.5},
  ] coordinates {(3, 0)};
  \addplot[
    only marks,
    mark=*,
    mark options={color=green, fill=green, scale=.5},
  ] coordinates {(4, 0)};
  \addplot[
    only marks,
    mark=*,
    mark options={color=green, fill=green, scale=.5},
  ] coordinates {(5, 0)};
  \addplot[
    only marks,
    mark=*,
    mark options={color=green, fill=green, scale=.5},
  ] coordinates {(6, 0)};
  \addplot[
    only marks,
    mark=*,
    mark options={color=green, fill=green, scale=.5},
  ] coordinates {(7, 0)};
  \addplot[
    only marks,
    mark=*,
    mark options={color=green, fill=green, scale=.5},
  ] coordinates {(8, 0)};
   \addplot[
    only marks,
    mark=*,
    mark options={color=green, fill=green, scale=.5},
  ] coordinates {(9, 0)};
  \addplot[
    only marks,
    mark=*,
    mark options={color=green, fill=green, scale=.5},
  ] coordinates {(0, 7)};
  \addplot[
    only marks,
    mark=*,
    mark options={color=green, fill=green, scale=.5},
  ] coordinates {(0, 0)};
  \end{axis}
\end{tikzpicture}
\end{center}
\caption{Graph of $f_{[7 \ 9]}(\mathbf{x})$}
  \label{fig:GraphSevenNine}
  \end{figure}
\begin{corollary}\label{cor:exact_blue_red}
There are exactly $\frac{(a-1)(b-1)}{2}$ blue points and exactly $\frac{(a-1)(b-1)}{2}$ red points.
\end{corollary}
\begin{proof}
Consider the region $D^+ \cup D^0$. By Pick's Theorem, the area of the triangle $D^+ \cup D^0$ is $Area(D^+ \cup D^0) = i + \frac{j}{2}-1$ where $i$ is the amount of interior lattice points and $j$ is the amount of boundary lattice points. So,
\begin{align*}
A(D^+ \cup D^0) &= i + \frac{j}{2}-1\\
\frac{a\cdot{b}}{2} &= i + \frac{a+b+1}{2}-1 \\
i &= \frac{ab-a-b-1}{2}+1 \\
i  &= \frac{(a-1)(b-1)}{2}
\end{align*}
Therefore, $\lvert D^+ \rvert$ is equal to $i=\frac{(a-1)(b-1)}{2}$. By Proposition \ref{prop:equal_blue_red}, $\lvert D^- \rvert$ also equals $\frac{(a-1)(b-1)}{2}$.
\end{proof}
For the special case of $n = 2$, \cite{sylvester1882} discovered an explicit formula for the Frobenius number of two relatively prime integers: $ g(a,b) = ab - a - b$.
\begin{theorem}\label{thm:sylvester_2} {\rm{(\cite{sylvester1882})}}
Let $a,b$ be relatively prime integers. Then the greatest number that is not representable as a nonnegative linear combination of $a, b$ is given by $g(a,b) = ab-a-b$.
\end{theorem}
\begin{proposition}\label{prop:f(x,y)_sylvester_2}
$g(a,b) = f_{[a \ b]}(b-1,-1)$.
\end{proposition}
\begin{proof}
By Theorem \ref{thm:sylvester_2}, $g(a,b) = ab - a - b$. By the aforementioned definition, $f_{[a \ b]}(b-1,-1)$ is equal to
\begin{align*}
  f_{[a \ b]}(b-1,-1) &= a(b-1) + b(-1)\\
    &= ab-a-b.
\end{align*}
\end{proof}
When the line $ax+by=ab$ is translated down, the first such line with a lattice point in $D^-$ would be $ax+by = ab-a-b$ at $(x,y) = (b-1,-1)$. Hence, the line $ax+by = ab-ax_0-by_0$ has no point in $D^+$ by Lemma \ref{lem:+diophantine}. Accordingly, the Frobenius number for the pair ${7, 9}$ is the point $(8,-1)$ in the northeast corner of $D^-$ which has a corresponding value of $f_{[7 \ 9]}(8,-1) = 47$. \cite{brauer1942} refines Sylvester's argument, and our previous results (Proposition \ref{prop:equal_blue_red} and Corollary \ref{cor:exact_blue_red}) are graphical demonstrations of Brauer's concepts, making his ideas and Sylvester's result more visually evident.
\begin{theorem}\label{thm:equal_brauer}
Let $a$ and $b$ be relatively prime positive integers. Then, every positive integer $c$, between $0$ and $ab$, not divisible by $a$ or by $b$ is representable in the form
\begin{equation}
  \label{eq:ax+by}
c = ax+by
\end{equation}
for $x, y > 0$ or
\begin{equation}
  \label{eq:ab-ax-by}
    c = ab-ax-by
\end{equation}
for $x, y > 0$. Additionally, $c$ is not simultaneously representable in both form $(\ref{eq:ax+by})$ and form $(\ref{eq:ab-ax-by})$.
\end{theorem}
\begin{proof}
Let $u$ be the solution to the congruence $au \equiv c \pmod{b}$ where $0 < u < b$. If $c = au+bv$ where $v > 0$ and $v \in \mathbb{Z}$, $c$ is in form $(\ref{eq:ab-ax-by})$: $ax+by$ where $x=u$ and $v=y$. When $v < 0$ and $v \in \mathbb{Z}$, then $v=-y$, and
\begin{align*}
c &= au+bv\\
&= ab-a(b-u)-by\\
&= ab-ax-by,\\
\end{align*}
where $x=b-u > 0$, and thus, $c$ is in form ($\ref{eq:ax+by}$).
If $c$ is simultaneously representable in both forms, $ab-ax-by = ax' + by'$ where $x,x',y,y' > 0$.  Then, $ab=a(x+x')+b(y+y')$, which is impossible since $x+x' > 0$ and $y + y' > 0$. Note that the lattice points in Theorem \ref{thm:equal_brauer} of the form (\ref{eq:ax+by}) and (\ref{eq:ab-ax-by}) are blue and red points, respectively.
\end{proof}
Just as Brauer classifies $c$ into two forms for a relatively prime pair $a$ and $b$, we classify $D$ into two categories: blue $(D^+)$ and red lattice points $(D^-)$. Graphically speaking, Theorem \ref{thm:equal_brauer} can be interpreted that the lattice points not labeled green are either in $D^+$, labeled blue, or in $D^-$, labeled red, but not in both $D^+$ and $D^-$, i.e. $D^+ \cap D^- = \emptyset$.
\section{Exceptional Values}\label{sec:Exceptional Values}
For three relatively prime integers $a$, $b$, and $c$, it is obvious that the Frobenius number $g(a,b,c) \leq g(a,b)$ because any nonnegative integer linear combination of ${a,b}$ is also a nonnegative integer linear combination of ${a,b,c}$. We will see that for relatively prime integers $a$ and $b$, for many $c$, the Frobenius number $g(a,b,c) = g(a,b)$. In what follows, we will concentrate on those $c$ such that $g(a,b,c) < g(a,b)$.
\begin{definition}\label{def:exceptional_defn}
For three relatively prime positive integers $a$, $b$ and $c$ where $a<b<c$, we define $c$ as an \textit{exceptional value} with respect to $a$ and $b$ when $g(a,b,c) < g(a,b)$. 
\end{definition}
Now, we characterize the exceptional values $c$.
\begin{proposition}\label{prop:exceptional_scenario}
$c$ is of the form $ab-ax-by$ for $x, y > 0$ if and only if $g(a,b,c) < g(a,b)$. In other words, a lattice point in $D$ is in $D^-$ if and only if its corresponding value $c$ is exceptional.
\end{proposition}
\begin{proof}
Assume $c = ab-ax-by$ for $x, y > 0$. Then,
\begin{align*}
ab &= c+ax+by \\
ab-a-b &= (x-1)a+(y-1)b+1\cdot{c} \\
g(a,b) &= (x-1)a+(y-1)b+1\cdot{c}
\end{align*}
where $x-1, y-1 \geq 0$. If $g(a,b,c)$ is equal to $g(a,b)=ab-a-b$, then $g(a,b,c)$ would be an nonnegative integer linear combination of ${a,b,c}$, which contradicts the definition of $g(a,b,c)$. Thus, $g(a,b,c) < g(a,b)$ if $c$ is of the form $ab-ax-by$ for $x,y > 0$. \par
Conversely, suppose $c$ is not of the form $ab-ax-by$, then $g(a,b,c)$ is not less than $g(a,b)$. As shown in Theorem \ref{thm:equal_brauer}, if $c$ is not of the form $ab-ax-by$, then $c$ must be representable in the form $ax+by$ for relatively prime positive integers $a$ and $b$ and $x,y > 0$. Furthermore, since $g(a,b,c)$ is always less than or equal to $g(a,b)$, $g(a,b,c) = g(a,b)$ if $g(a,b,c)$ is not less than $g(a,b)$. Therefore, if $c$ is in form $(\ref{eq:ax+by})$, then $g(a,b,c) = g(a,b)$. If $c$ is a nonnegative linear combination of $a$ and $b$, then a nonnegative linear combination of $a$, $b$, and $c$ is equivalent to a nonnegative linear combination of $a$ and $b$. Thus, by the definition of the Frobenius number, $g(a,b,c) = g(a,b)$.
\end{proof}
Since $c$ must be of the form $ab-ax-by$ for $x,y > 0$ for $g(a,b,c) < g(a,b)$, all the exceptional values $c$ can be visualized as the red dots strictly above the dotted orange line $ax+by = b$ as shown in Figure \ref{fig:GraphSevenNine}, taking into consideration our assumption that $b < c$. Let these lattice points be denoted by $D^*$, the set of lattice points in $D^-$ that lie strictly above the line $ax+by = b$. 
\begin{theorem}\label{thm:exceptional_num}
For relatively prime positive integers $a$ and $b$, there are exactly $\frac{(a-3)(b-1)}{2} + \lfloor \frac{b}{a} \rfloor$ exceptional values.
\end{theorem}
\begin{proof}
As noted in Corollary \ref{cor:exact_blue_red}, there are exactly $\frac{(a-1)(b-1)}{2}$ red points. However, we need to exclude all points $(x,y)$ with $f_{[a \ b]}({x,y}) \leq b$ by Definition \ref{def:exceptional_defn}, but we need to not exclude the point $(0,1)$ with a value of $f_{[a \ b]}(0,1) = b$ because it is not a red point. Note that the  "cut-off" for exceptional values can be represented by the equation $ax+by = b$ which will intercept the $x$-axis at $(\frac{b}{a},0)$. So, we need to not exclude the points {$(1,0), (2,0) \cdots (\lfloor \frac{b}{a} \rfloor, 0)$} because they are also not exceptional values. So, by Inclusion-exclusion principle, the final expression for the number of exceptional values for relatively prime integers $a$ and $b$ is
\begin{align*}
\frac{(a-1)(b-1)}{2} - [(b-1)- \lfloor \frac{b}{a} \rfloor] &= \frac{ab-a-b+1}{2} - \frac{2b}{2} + \frac{2}{2} \\
&= \frac{ab-a-3b+3}{2} + \lfloor \frac{b}{a} \rfloor \\ 
&= \frac{(a-3)(b-1)}{2} + \lfloor \frac{b}{a} \rfloor \\
\end{align*}
\qedhere
\end{proof}
\section{Evaluating $g(a,b,c)$}\label{sec:Evaluating g(a,b,c)}
For relatively prime integers $a$, $b$, and $c$, we create a graph to determine $g(a,b,c)$, a graph that reveals a pattern upon careful observation. The figure and its key points/lines are shown below.
  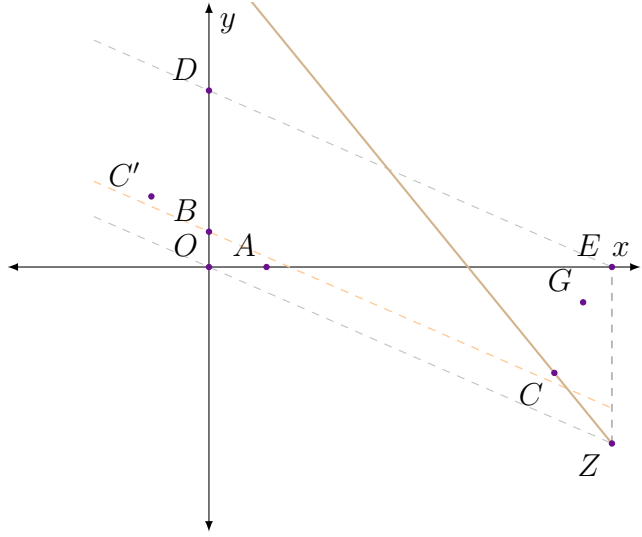
\begin{figure}[H]
\begin{center}    
\begin{tikzpicture}
\begin{axis}[
    axis lines=center,
    xlabel=$x$,
    ylabel=$y$,
    xmin=-3,
    xmax=7,
    ymin=-7,
    ymax=7,
    xtick=\empty, 
    ytick=\empty, 
    grid=both,
    grid style={line width=0.2pt, draw=gray!30},
    major grid style={line width=0.6pt,draw=gray!60},
    ticklabel style={font=\tiny},
    enlargelimits={abs=0.5},
    axis line style={latex-latex},
  ]
  \addplot[
    domain=-10:10,
    samples=2,
    color=black!50,
    dashed
  ] coordinates {(7, -5) (7, 0)};
  \addplot[
    domain=-2:7,
    samples=100,
    color=black!25,
    dashed,
  ]{((35 - 5*x)/7)};
  \addplot[
    domain=-2:7,
    samples=100,
    color=black!25,
    dashed,
  ]{((-5*x)/7)};
  \addplot[
    domain=-2:7,
    samples=100,
    color=brightorange!50,
    dashed,
  ]{((-5*x)/7+1)};
  \addplot[
    domain=-2:7,
    samples=100,
    color=brightorange!50,
    dashed,
  ]{((-5*x)/7+1)};

  \addplot[
    only marks,
    mark=*,
    mark options={color=vividpurple, fill=vividpurple, scale=.5},
  ] coordinates {(0, 0)};
  \node[above left, font=\normalsize] at (axis cs:0,0) {$O$};
  \addplot[
    only marks,
    mark=*,
    mark options={color=vividpurple, fill=vividpurple, scale=.5},
  ] coordinates {(1, 0)};
  \node[above left, font=\normalsize] at (axis cs:1,0) {$A$};
  \addplot[
    only marks,
    mark=*,
    mark options={color=vividpurple, fill=vividpurple, scale=.5},
  ] coordinates {(0, 1)};
  \node[above left, font=\normalsize] at (axis cs:0,1) {$B$};
  \addplot[
    only marks,
    mark=*,
    mark options={color=vividpurple, fill=vividpurple, scale=.5},
  ] coordinates {(7, 0)};
  \node[above left, font=\normalsize] at (axis cs:7,0) {$E$};
  \addplot[
    only marks,
    mark=*,
    mark options={color=vividpurple, fill=vividpurple, scale=.5},
  ] coordinates {(0, 5)};
  \node[above left, font=\normalsize] at (axis cs:0,5) {$D$};
  \addplot[
    only marks,
    mark=*,
    mark options={color=vividpurple, fill=vividpurple, scale=.5},
  ] coordinates {(7, -5)};
  \node[below left, font=\normalsize] at (axis cs:7,-5) {$Z$};
  \addplot[
    only marks,
    mark=*,
    mark options={color=vividpurple, fill=vividpurple, scale=.5},
  ] coordinates {(6.5,-1)};
  \node[above left, font=\normalsize] at (axis cs:6.5,-1) {$G$};
  \addplot[
    only marks,
    mark=*,
    mark options={color=vividpurple, fill=vividpurple, scale=.5},
  ] coordinates {(6,-3)};
  \node[below left, font=\normalsize] at (axis cs:6,-3) {$C$};
  \addplot[
    only marks,
    mark=*,
    mark options={color=vividpurple, fill=vividpurple, scale=.5},
  ] coordinates {(-1,2)};
  \node[above left, font=\normalsize] at (axis cs:-1,2) {$C'$};
  \addplot[domain=-2:7, color=tan, thick]{-2*x+9};  
\end{axis}
\end{tikzpicture}
\end{center}
\caption{General Graph for Section \ref{sec:Evaluating g(a,b,c)}}
  \label{fig:General}
  \end{figure}
\begin{center}
\begin{tabular}{|l|p{8cm}|}
    \hline
    \textbf{Point} & \textbf{$f_{[a \ b]}(x,y)$} \\
    \hline
    $O(0,0)$ & $f_{[a \ b]}(O) = 0$ \\
    \hline
    $A(1,0)$ & $f_{[a \ b]}(A) = a$ \\
    \hline
    $B(0,1)$ & $f_{[a \ b]}(B) = b$ \\
    \hline
    $E(0,b)$, $D(a,0)$ & $f_{[a \ b]}(D) = f_{[a \ b]}(E) = ab$ \\
    \hline
    $C(b-h,l-a)$, $C'(-h,l)$ & $f_{[a \ b]}(C) = f_{[a \ b]}(C')=-ha+lb = c$ for $b < f_{[a \ b]}(C) \leq g(a,b)$. Note that the corresponding values for $C$ and $C'$ are exceptional values. \\
    \hline
    $Z(b,-a)$ & $f_{[a \ b]}(Z) = 0$ \\
    \hline
    $G(b-1,-1)$ & $f_{[a \ b]}(G) = g(a,b)$ by Proposition \ref{prop:f(x,y)_sylvester_2}\\
    \hline
\end{tabular}
\end{center}
We want $X(b-m_0,-n_0)$ among the red lattice points, which has a corresponding value $f_{[a \ b]}({X}) = ab-am_{0}-bn_{0}$ \emph{not} expressible as a nonnegative integer linear combination of relatively prime integers ${a,b,c}$. Thus, $f_{[a \ b]}({X}) \leq g(a,b,c)$. The greatest such value will be equal to $g(a,b,c)$.
Let $ab-m_{0}a-n_{0}b$ be expressible in the form $c_{1}a+c_{2}b+c_{3}c$ where $c = -ha + lb$.
\begin{align*}
ab-m_{0}a-n_{0}b &= c_{1}a+c_{2}b+c_{3}c \\
(-m_0)a + (a-n_0)b &= (c_1 - hc_3)a + (c_2 + lc_3)b
\end{align*}
Thus,
\begin{equation}\label{eq:key_general}
\begin{aligned}
    c_1 - hc_3 &= -m_0 \\
    c_2 + lc_3 &= a - n_0
\end{aligned}
\end{equation}
Further simplifying:
\begin{align*}
    lc_1 - lhc_3 &= -lm_0 \\
    hc_2 - lhc_3 &= (a - n_0)h \\
    \intertext{Summing the equations:}
    lc_1+hc_2 &= ah - (lm_0+hn_0)
\end{align*}
\begin{equation}\label{eq:key_general_sub}
    lc_1+hc_2 = ah - f_{[l \ h]}(m_0,n_0)
\end{equation}
From Equations \ref{eq:key_general_sub}, it's clear that if $ah - f_{[l \ h]}(m_0,n_0)$ is not expressible as a nonnegative integer linear combination of $l$ and $h$, then $X$ is also not expressible as a nonnegative linear combination of $a$, $b$, and $c$. For the various not expressible points in $D^-$, we classify them into three categories.
\begin{enumerate}[label={},itemsep=0pt]
    \item Case 1: If $f_{[l \ h]}(m_0,n_0) > ah$, then $X$ is not expressible as a nonnegative linear combination of ${a,b,c}$ since $lc_1 + hc_2 = ah-f_{[l \ h]}(m_0,n_0) < 0$.
    \item Case 2: $lc_1 + hc_2 = ah-f_{[l \ h]}(m_0,n_0)$ where $ah-f_{[l \ h]}(m_0,n_0)$ is a positive integer that is not expressible as a nonnegative linear combination of ${l,h}$, according to Theorem \ref{thm:sylvester_2}.
    \item Case 3: Equation \ref{eq:key_general_sub} has nonnegative integer solutions for $c_1$ and $c_2$, but Equation \ref{eq:key_general} has no nonnegative integer solution for $c_3$.
\end{enumerate}
Utilizing these cases, we now provide a grand outline for proof going forward.
\begin{center}
\begin{tabular}{|l|p{4cm}|}
    \hline
    $h$ & $a$ \\
    \hline
    $h \equiv 0 \pmod{l}$ & $a \equiv 0 \pmod{l}$ \\
    \hline
    $h \equiv 0 \pmod{l}$ & $a \equiv 1 \pmod{l}$ \\
    \hline
    $h \equiv 0 \pmod{l}$ & $a \not\equiv 0,1 \pmod{l}$ \\
    \hline
    $h \not\equiv 0 \pmod{l}$ & $a \equiv 0 \pmod{l}$ \\
    \hline
    $h \not\equiv 0 \pmod{l}$ & $a \equiv 1 \pmod{l}$ \\
    \hline
    $h \not\equiv 0 \pmod{l}$ & $a \not\equiv 0,1 \pmod{l}$ \\
    \hline
\end{tabular}
\end{center}
\subsection{The Three Cases}
Before we begin, we first elaborate on the three aforementioned cases.
\subsubsection{Case 1}\label{sec:Case 1}
Relatively prime integers ${a,b,c}$ will always have a point $F$ lying below line $CZ$ whose corresponding value is not expressible as a nonnegative integer linear combination of ${a,b,c}$. \par
We characterize the point $F$. As aforementioned, Case 1 occurs when $f_{[l \ h]}(m_0,n_0) > ah$. Therefore, we set up the $u$-$v$ coordinate plane where $u=b-x$ and $v=-y$ i.e. take $E$ as the origin of the $u$-$v$ plane, take the $-x$ direction as the positive $u$ direction, and $-y$ as the positive $v$ direction. The lattice points in the first quadrant outside of the region enclosed by the line $f_{[l \ h]}(m_0,n_0) = ah$, $u$-axis, and $v$-axis are not expressible. The point in the $u$-$v$ plane that will generate the largest $f_{[a \ b]}(F)$ will be $F'(x,1)$ since its counterpart, $F(b-x,-1)$, will be the first lattice point outside of line $CZ$ hit by the line $ax+by=ab$ and its downards translations. More specifically, the value of $m_0 = (\lfloor \frac{ah-h}{l} \rfloor + 1)$ or the next integer strictly above the integer value where the line $lm_{0}+hn_{0} = ah$ intersects $y=1$ as seen below.
\begin{figure}[H]
\begin{center}
    \begin{tikzpicture}[scale=0.8]
    \begin{axis}[
    axis lines=center,
    xlabel=$u$,
    ylabel=$v$,
    xmin=0,
    xmax=9,
    ymin=0,
    ymax=6,
    xticklabel style={above},
    xtick={0,1, ...,9,10,11,12,13,14,15,16,17,18},
    ytick={-7, ..., 7},
    grid=both,
    grid style={line width=0.2pt, draw=gray!30},
    major grid style={line width=0.6pt,draw=gray!60},
    ticklabel style={font=\small},
    enlargelimits={abs=0.5},
    axis line style={latex-latex},
  ]
  \addplot[
    domain=0:9,
    samples=100,
    color=tan,
    thick,
  ]{(6-x)};
   \addplot[only marks, mark=*, mark options={color=red, fill=red, scale=0.5}] coordinates {(1, 1)};   
   \addplot[only marks, mark=*, mark options={color=black, fill=black, scale=0.5}] coordinates {(0, 0)};  
   \addplot[only marks, mark=*, mark options={color=black, fill=black, scale=0.5}] coordinates {(0, 6)}; 
   \addplot[only marks, mark=*, mark options={color=red, fill=red, scale=0.5}] coordinates {(6, 1)}; 
\addplot[only marks, mark=*, mark options={color=red, fill=red, scale=0.5}] coordinates {(2, 4)}; 
   \node[below right, font=\normalsize] at (axis cs:0,6) {$Z'$};
    \node[below right, font=\normalsize] at (axis cs:0,0) {$E'$};
     \node[above right, font=\normalsize] at (axis cs:6,1) {$F'$};
    \node[below left, font=\normalsize] at (axis cs:1,1) {$G'$};
    \node[below left, font=\normalsize] at (axis cs:2,4) {$C'$};
  \end{axis}
  \end{tikzpicture}
    \caption{$f_{[l \ h]}(m_0,n_0)=ah$ for $h=2$, $l=2$}
\end{center}
\end{figure}
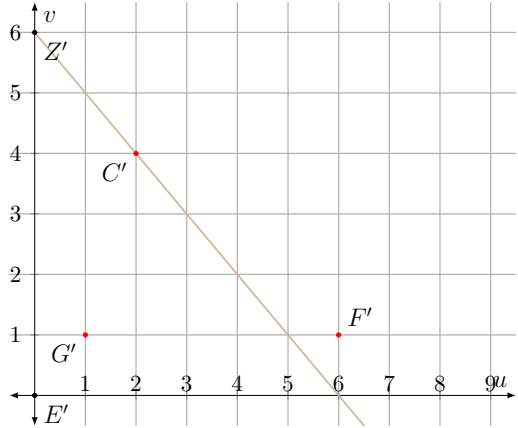
\subsubsection{Case 2}\label{sec: Case 2}
This occurs when $lc_1+hc_2$ is equal to a "small" positive integer that is not expressible as a nonnnegative integer linear combination of $l$ and $h$: $lh - lm_0-hn_0$ for $m_0,n_0 > 0$. Thus, $f_{[l \ h]}(m_0,n_0) = ah-(lh-lm_0-hn_0)$. By the previous equality, Equation \ref{eq:key_general_sub} becomes $lc_1 + hc_2 = ah-f_{[l \ h]}(m_0,n_0) = ah - (ah-(lh-lm_0-hn_0)) = lh - lm_0 - hn_0$. The largest is when $m_0 = n_0 = 1$, but we later observe that this is not always the most optimal at realizing $g(a,b,c)$. Points that fall under Case 2 are labeled as $P$.
\subsubsection{Case 3}\label{sec:Case 3}
Case 3 happens when Equation \ref{eq:key_general_sub} is solvable but Equation \ref{eq:key_general} is not. That is, $c_1$ and $c_2$ are nonnegative integers, but $c_3$ is not a nonnegative integer. We label these points as $Q$ points. This scenario will become more clear as we observe more concrete examples later. 

\subsection{Candidates for Frobenius Number for $h \equiv 0 \pmod{l}$}\label{subsec: h=0}
If $h \equiv 0 \pmod{l}$, then we have the graph $f_{[l \ h]}(m_0,n_0) = lm_0 + hn_0 = ah$, which can be simplified to $m_0 + \frac{h}{l}n_0 = a\frac{h}{l}$ since $h \equiv 0 \pmod{l}$. Therefore, a point in Case 2 does not exist since $g(1,\frac{h}{l})$ does not exist. Let $a$ be representable in the unique form
\begin{equation}
a=ql+r
\end{equation}
where $l$ is the divisor, $q$ is the quotient, and $r$ is the remainder. Note that $q > 0$ since $l < a$. Since there is no $P$ point, we either have a $F$ point or $Q$ point. $F'$ is located at $(\lfloor \frac{ah-h}{l} \rfloor + 1,1)$ by \ref{sec:Case 1}. We now determine the location of $Q'$. We note that there are two possible optimal locations of $Q'$. We label a canonical $Q$ as $Q'_1$, and we label an anomaly $Q$ as $Q'_2$.
\begin{proposition}\label{prop:a/b < l/h}
$-\frac{a}{b} > -\frac{l}{h}$
\end{proposition}
\begin{proof}
By assumption, $b < c = -ha+lb$. This gives us $\frac{b}{a} > \frac{h}{l-1}$. Note that $a,b,c,h$ and $l$ are all positive integers, so $-\frac{a}{b} > -\frac{l-1}{h}$ or $-\frac{a}{b} > -\frac{l}{h}$. Note that if $l \leq 1$ and $l\geq a$, $c$ cannot be an exceptional value, so $1 < l < a$.
\end{proof}
\begin{proposition}\label{prop:Possible_Q_Location}
For $h \equiv 0 \pmod{l}$, there is a $Q$ located at $Q'_1((q-1)h+1,r+1)$ and a possible $Q$ located at $Q'_2(qh+1,1)$.
\end{proposition}
\begin{proof}
In equation \ref{eq:key_general}, let $c_2 = l+r-n_0$. Therefore, $c_3 = q-1$, and $c_1 = -m_0 + hq-h$. Since $c_1$ must be greater than or equal to 0, there will always be a solution when $m_0 \leq hq-h$, if $c_2, c_3 \geq 0$. More specifically, if $Q(u,v)$ is solvable for Equation \ref{eq:key_general} for nonnegative integers $c_1,c_2,c_3$, then $Q(u-n,v)$ is solvable for Equation \ref{eq:key_general} for nonnegative integers $c_1+n,c_2,c_3$. Furthermore, $Q(u+hn,v-ln)$ is solvable for Equation \ref{eq:key_general} for nonnegative integers $c_1,c_2,c_3-n$ and that $c_3-n < 0$ only occurs outside of the region bounded by the $m_0$ axis, $n_0$ axis, and $lm_0+hn_0 = ah$. An example can be seen in Figure \ref{fig:h=0,a=0} where the values of $c_1,c_2,c_3$ are in the format $(c_1,c_2,c_3)$. As displayed, $R_3$ is solveable, and so is $R_1$.
\par We now show that $Q'_1((q-1)h+1,r+1)$ is indeed a $Q$ point. From Equation \ref{eq:key_general}, we have:
\begin{align*}
    c_1 - hc_3 &= -(q-1)h-1 \\
    c_2 + lc_3 &= ql-1
\end{align*}
Note that for the equation $c_2 + lc_3 = ql-1$, $c_3$ must be less than $q$, which can be expressed as $c_3 = q-n$ for $n > 0$. Thus, $c_2 = nl-1$ and $c_1 = (-n+1)h-1 < 0$. Therefore, $Q'_1$ is a $Q$ point.
\par We don't need to consider the points in the same triangular region as $Q'_1$ that are above or to the right of it since $Q'_1$ will always be hit first by the line $ax+by = ab$ and its downward parallel translations by observation. Note that if $Q'$ has coordinates $(u,v)$ which is not solvable for Equation \ref{eq:key_general}, then $(u+h, v-l)$ will also not be solvable. We do not need to consider the lattice points that lie on the line with slope $-\frac{l}{h}$ that goes through point $Q'$ since $Q'(u,v)$ will always be hit first by the line $ax+by = ab$ and its downward parallel translations by Proposition \ref{prop:a/b < l/h}. 
\par If we let $c_2 = r-n_0$, we get $c_3 = q$ and $c_1 = qh - m_0 \geq 0$ by following a similar process from before. Thus, we obtain a bound of $m_0 \leq qh$. However, note that the amount of lattice points to the left of $m_0 = qh$ that are solutions is a lot more limited since $c_2 = r-n_0 \geq 0$ as shown by the green region in Figure \ref{fig:h=0,a=1} and Figure \ref{fig:h=0,a=3}. Figure \ref{fig:h=0,a=0} shows the minimum amount of a green region with solutions, which is empty, while Figure \ref{fig:h=0,a=3} shows the maximum amount of green region with solutions: a rectangle with height $r=l-1$ and a width of $qh$. When $a \equiv 0 \pmod l$, this region doesn't exist as shown in Figure \ref{fig:h=0,a=0}. We now show that $Q'_2(qh+1,1)$ is indeed a $Q$ point: 
\begin{align*}
    c_1 - hc_3 &= -qh-1 \\
    c_2 + lc_3 &= a-1
\end{align*}
Note that for the equation $ c_1 - hc_3 = -qh-1$, $c_3$ must be greater than $q$, otherwise $c_1$ is negative. Therefore, if $c_3$ is in the form $q+n$ for $n>0$, then $c_1 = nh-1$ and $c_2 = r-1-ln < 0$. Again, we do not need to consider the points in the same triangular region as $Q'_2$ that are above or the right of it. 
\par It's important to note that $Q'_2(qh+1,1)$ occurs outside of the tan line $lm_0 + hn_0 = ah$ when $a \equiv 0 \pmod{l}$ or $a \equiv 1 \pmod{l}$, and therefore, it should be considered an $F'_2$ point—the $u$ coordinate of $F'$ is $\lfloor \frac{ah-h}{l} \rfloor + 1 = \frac{(al+r)h-h}{l} = qh+\frac{rh}{l}-\frac{h}{l}+1$ which is always less than or equal to the $u$ coordinate of $Q'_{1}$ for $a \equiv 0 \pmod{l}$ or $a \equiv 1 \pmod{l}$. However, we already have the most optimal $F'$, so we can neglect these scenarios. An example can be seen in Figure \ref{fig:h=0,a=0} and Figure \ref{fig:h=0,a=1}. Otherwise, we must consider $Q'_2$ when $a \not\equiv 0,1 \pmod{l}, h \equiv 0 \pmod{l}$, which we do in Section \ref{subsubsec:h=0,a=2+}.
\end{proof}
  \begin{figure}[H]
\begin{center}    
\begin{minipage}{0.33\textwidth}
\begin{tikzpicture}[scale=0.6]
\begin{axis}[
    axis lines=center,
    xlabel=$m_0$,
    ylabel=$n_0$,
    xmin=0,
    xmax=14,
    ymin=0,
    ymax=13,
    xticklabel style={above},
    xtick={0,1, ...,9,10,11,12,13,14},
    ytick={0,1, ...,9,10,11,12,13},
    grid=both,
    grid style={line width=0.2pt, draw=gray!30},
    major grid style={line width=0.6pt,draw=gray!60},
    ticklabel style={font=\tiny},
    enlargelimits={abs=0.5},
    axis line style={latex-latex},
  ]
  \addplot[
    domain=-10:10,
    samples=2,
    color=red,
    thick
  ] coordinates {(0, 8) (0,12)};
  \addplot[
    domain=-10:10,
    samples=2,
    color=red,
    thick
  ] coordinates {(0, 8) (4,8)};
  \addplot[
    domain=-10:10,
    samples=2,
    color=red,
    thick
  ] coordinates {(4, 8) (4,4)};
  \addplot[
    domain=-10:10,
    samples=2,
    color=red,
    thick
  ] coordinates {(4, 4) (8,4)};
  \addplot[
    domain=-10:10,
    samples=2,
    color=red,
    thick
  ] coordinates {(8, 4) (8,0)};
  \addplot[
    domain=-10:10,
    samples=2,
    color=red,
    thick
  ] coordinates {(2, 2) (6,2)};
  \addplot[
    domain=-10:10,
    samples=2,
    color=red,
    thick
  ] coordinates {(2, 6) (2,2)};
  \addplot[
    domain=-10:10,
    samples=2,
    color=red,
    thick
  ] coordinates {(8, 0) (12,0)};
  \addplot[
    only marks,
    mark=*,
    mark options={color=vividpurple, fill=vividpurple, scale=.5},
  ] coordinates {(9,1)};
  \node[above right, font=\normalsize] at (axis cs:9,1) {$Q'_1$};
  \addplot[
    only marks,
    mark=*,
    mark options={color=vividpurple, fill=vividpurple, scale=.5},
  ] coordinates {(5,5)};
  \node[above right, font=\normalsize] at (axis cs:5,5) {$Q'_{0}$};
  \addplot[
    only marks,
    mark=*,
    mark options={color=vividpurple, fill=vividpurple, scale=.5},
  ] coordinates {(1,9)};
  \node[above right, font=\normalsize] at (axis cs:1,9) {$Q'_{-1}$};
  \addplot[domain=0:12, color=tan, thick]{-x+12};  
  \addplot[domain=0:12, color=vividpurple, dashed]{-x+10};
  \filldraw[blue!20, opacity=0.5] (axis cs:0, 0) rectangle (axis cs:8, 4);
  \draw[->, thick] (axis cs:8, 7) -- (axis cs:8, 4);
  \node[above, font=\normalsize] at (axis cs:8, 7) {$hq-h$};

  \addplot[
    only marks,
    mark=*,
    mark options={color=vividpurple, fill=vividpurple, scale=.5},
  ] coordinates {(6,2)};
  \addplot[
    only marks,
    mark=*,
    mark options={color=vividpurple, fill=vividpurple, scale=.5},
  ] coordinates {(2,6)};
  \addplot[
    only marks,
    mark=*,
    mark options={color=vividpurple, fill=vividpurple, scale=.5},
  ] coordinates {(13,1)};
  \node[above right, font=\normalsize] at (axis cs:13,1) {$Q'_{2}$};
   \addplot[
    only marks,
    mark=*,
    mark options={color=vividpurple, fill=vividpurple, scale=.5},
  ] coordinates {(12,1)};
  \node[above left, font=\normalsize] at (axis cs:13,1) {$F'$};
  \addplot[
    only marks,
    mark=*,
    mark options={color=vividpurple, fill=vividpurple, scale=.5},
  ] coordinates {(4,2)};
  \node[above, font=\normalsize] at (axis cs:2,7) {$R_1$};
  \node[above, font=\normalsize] at (axis cs:4,2) {$R_2$};
  \node[above, font=\normalsize] at (axis cs:6,2) {$R_3$};
  \node[above, font=\small] at (axis cs:2,6) {$(2,2,1)$};
  \node[below, font=\small] at (axis cs:4,2) {$(4,2,2)$};
  \node[below, font=\small] at (axis cs:6,2) {$(2,2,2)$};
\end{axis}
\end{tikzpicture}
\caption{$h = 4\equiv 0, a =12 \equiv 0 \pmod{l=4}$}
\label{fig:h=0,a=0}
\end{minipage}\hspace{1cm}
\begin{minipage}{0.33\textwidth}
\begin{tikzpicture}[scale=0.6]
\begin{axis}[
    axis lines=center,
    xlabel=$m_0$,
    ylabel=$n_0$,
    xmin=0,
    xmax=15,
    ymin=0,
    ymax=13,
    xticklabel style={above},
    xtick={0,1, ...,9,10,11,12,13,14,15},
    ytick={0,1, ...,9,10,11,12,13},
    grid=both,
    grid style={line width=0.2pt, draw=gray!30},
    major grid style={line width=0.6pt,draw=gray!60},
    ticklabel style={font=\tiny},
    enlargelimits={abs=0.5},
    axis line style={latex-latex},
  ]
  \addplot[
    domain=-10:10,
    samples=2,
    color=red,
    thick
  ] coordinates {(0, 9) (0,13)};
  \addplot[
    domain=-10:10,
    samples=2,
    color=red,
    thick
  ] coordinates {(0, 9) (4,9)};
  \addplot[
    domain=-10:10,
    samples=2,
    color=red,
    thick
  ] coordinates {(4, 9) (4,5)};
  \addplot[
    domain=-10:10,
    samples=2,
    color=red,
    thick
  ] coordinates {(4, 5) (8,5)};
  \addplot[
    domain=-10:10,
    samples=2,
    color=red,
    thick
  ] coordinates {(8, 1) (8,5)};
  \addplot[
    domain=-10:10,
    samples=2,
    color=red,
    thick
  ] coordinates {(8, 1) (12,1)};
  \addplot[
    only marks,
    mark=*,
    mark options={color=vividpurple, fill=vividpurple, scale=.5},
  ] coordinates {(9,2)};
  \node[above right, font=\normalsize] at (axis cs:9,2) {$Q'_1$};
  \addplot[
    only marks,
    mark=*,
    mark options={color=vividpurple, fill=vividpurple, scale=.5},
  ] coordinates {(5,6)};
  \node[above right, font=\normalsize] at (axis cs:5,6) {$Q'_{0}$};
  \addplot[
    only marks,
    mark=*,
    mark options={color=vividpurple, fill=vividpurple, scale=.5},
  ] coordinates {(1,10)};
  \node[above right, font=\normalsize] at (axis cs:1,10) {$Q'_{-1}$};
  \addplot[domain=0:13, color=tan, thick]{-x+13};  
  \addplot[domain=0:13, color=vividpurple, dashed]{-x+11};
  \filldraw[blue!20, opacity=0.5] (axis cs:0, 1) rectangle (axis cs:8, 5);
  \draw[->, thick] (axis cs:8, 8) -- (axis cs:8, 5);
  \node[above, font=\normalsize] at (axis cs:8,8) {$qh-h$};
   \filldraw[green!20, opacity=0.5] (axis cs:0, 0) rectangle (axis cs:12,1);
  \draw[->, thick] (axis cs:12,4) -- (axis cs:12,1);
  \node[above, font=\normalsize] at (axis cs:12,4) {$qh$};
  \addplot[
    only marks,
    mark=*,
    mark options={color=vividpurple, fill=vividpurple, scale=.5},
  ] coordinates {(13,1)};
  \node[above right, font=\normalsize] at (axis cs:13,1) {$F' (Q'_{2})$};

\end{axis}
\end{tikzpicture}
\caption{$h = 4 \equiv 0, a = 13 \equiv 1 \pmod{l=4}$}
\label{fig:h=0,a=1}
\end{minipage}\hspace{1cm}
\begin{minipage}{0.33\textwidth}
\begin{tikzpicture}[scale=0.6]
\begin{axis}[
    axis lines=center,
    xlabel=$m_0$,
    ylabel=$n_0$,
    xmin=0,
    xmax=16,
    ymin=-1,
    ymax=16,
    xticklabel style={above},
    xtick={0,1, ...,9,10,11,12,13,14,15,16},
    ytick={-1,0,1, ...,9,10,11,12,13,14,15,16},
    grid=both,
    grid style={line width=0.2pt, draw=gray!30},
    major grid style={line width=0.6pt,draw=gray!60},
    ticklabel style={font=\tiny},
    enlargelimits={abs=0.5},
    axis line style={latex-latex},
  ]
  \addplot[
    domain=-10:10,
    samples=2,
    color=red,
    thick
  ] coordinates {(0, 11) (0,15)};
  \addplot[
    domain=-10:10,
    samples=2,
    color=red,
    thick
  ] coordinates {(0, 11) (4,11)};
  \addplot[
    domain=-10:10,
    samples=2,
    color=red,
    thick
  ] coordinates {(4, 11) (4,7)};
  \addplot[
    domain=-10:10,
    samples=2,
    color=red,
    thick
  ] coordinates {(4, 7) (8,7)};
  \addplot[
    domain=-10:10,
    samples=2,
    color=red,
    thick
  ] coordinates {(8, 7) (8,3)};
  \addplot[
    domain=-10:10,
    samples=2,
    color=red,
    thick
  ] coordinates {(8, 3) (12,3)};
   \addplot[
    domain=-10:10,
    samples=2,
    color=red,
    thick
  ] coordinates {(12, -1) (12,3)};
  \addplot[
    domain=-10:10,
    samples=2,
    color=red,
    thick
  ] coordinates {(12, -1) (16,-1)};
  \addplot[
    only marks,
    mark=*,
    mark options={color=vividpurple, fill=vividpurple, scale=.5},
  ] coordinates {(9,4)};
  \node[above right, font=\normalsize] at (axis cs:9,4) {$Q'_1$};
  \addplot[
    only marks,
    mark=*,
    mark options={color=vividpurple, fill=vividpurple, scale=.5},
  ] coordinates {(5,8)};
  \node[above right, font=\normalsize] at (axis cs:5,8) {$Q'_{0}$};
  \addplot[
    only marks,
    mark=*,
    mark options={color=vividpurple, fill=vividpurple, scale=.5},
  ] coordinates {(1,12)};
  \node[above right, font=\normalsize] at (axis cs:1,12) {$Q'_{-1}$};
  \addplot[domain=0:16, color=tan, thick]{-x+15};  
  \addplot[domain=0:16, color=vividpurple, dashed]{-x+13};
  \filldraw[blue!20, opacity=0.5] (axis cs:0, 3) rectangle (axis cs:8, 7);
  \draw[->, thick] (axis cs:8, 10) -- (axis cs:8, 7);
  \node[above, font=\normalsize] at (axis cs:8,10) {$qh-h$};
   \filldraw[green!20, opacity=0.5] (axis cs:0, 0) rectangle (axis cs:12,3);
  \draw[->, thick] (axis cs:12,6) -- (axis cs:12,3);
  \node[above, font=\normalsize] at (axis cs:12,6) {$qh$};
  \addplot[
    only marks,
    mark=*,
    mark options={color=vividpurple, fill=vividpurple, scale=.5},
  ] coordinates {(13,1)};
  \node[above right, font=\normalsize] at (axis cs:13,1) {$Q'_{2}$};
  \addplot[
    only marks,
    mark=*,
    mark options={color=vividpurple, fill=vividpurple, scale=.5},
  ] coordinates {(15,1)};
  \node[above right, font=\normalsize] at (axis cs:15,1) {$F'$};
\end{axis}
\end{tikzpicture}
\caption{$h = 4 \equiv 0, a = 15 \equiv 3 \pmod{l=4}$}
\label{fig:h=0,a=3}
\end{minipage}
\end{center}
  \end{figure}\label{fig:h=0}
\subsection{Value Comparisons for $h \equiv 0 \pmod{l}$}\label{subsec: value comparisons for h=0}
For each of the points $F'$, $Q'_{1}$, and $Q'_{2}$, the one with the smallest corresponding value in the $u$-$v$ plane $f_{[a \ b]}({u,v})$, and thus the greatest corresponding value in the $x$-$y$ plane $f_{[a \ b]}({x,y})$ is highly dependent on the slope $-\frac{a}{b}$. For example, in Figure \ref{fig:h=0,a=1} or the case $a \equiv 1 \pmod{l}$, if $-\frac{a}{b}$ is close to $-1$, $Q'_1$ will clearly be hit before $F'$ with $ax+by=ab$ and its downward parallel translations. However, if $-\frac{a}{b}$ approaches 0, then the lattice point $F'$ will be hit first. A similar argument can be made between $Q'_{1}$ and $Q'_{2}$ when $a \not\equiv 0,1 \pmod {l}$. Therefore, it's necessary to determine what values for $a$ and $b$ determine the switch from $F'$ to $Q'$ or $Q'_{1}$ to $Q'_{2}$. We can determine these conditions by analyzing the corresponding values at each possible candidate for $g(a,b,c)$.
\subsubsection{$h \equiv 0 \pmod{l},a \equiv 0 \pmod {l}$}\label{subsubsec:h=0,a=0}
In Proposition \ref{prop:Possible_Q_Location}, we noted that for $a \equiv 0 \pmod {l}$ or $a \equiv 1 \pmod {l}$, $Q'_2$ lies outside the line $lm_0 + hn_0 = ah$ and need not be considered. So, we only need to consider $Q'_1$ and $F'$. $Q'_{1}$ is located at $((q-1)h+1,r+1) = ((q-1)h+1,1)$ since $r=0$. $F'$ is located at $(\lfloor \frac{ah-h}{l} \rfloor +1 , 1) = (qh - \frac{h}{l}+1,1)$ since $(h \equiv 0 \pmod {l})$. Therefore, $Q'_1$ will always have a smaller value than $F'$, and $Q_1$ always has a greater value than $F$. This is the most trivial case.
\subsubsection{$h \equiv 0 \pmod{l},a \equiv 1 \pmod {l}$}
Following a similar process to Section \ref{subsubsec:h=0,a=0}, we consider the values of $Q'_1$ and $F'$:
\begin{align}\label{eq: h=0, a=1}
 f_{[a \ b]}(Q'_1) &\substack{>\\<}   f_{[a \ b]}(F') \notag \\
    a((q-1)h+1) + b(r+1) &\substack{>\\<} a(\frac{ah-h}{l}+1) + b(1) \notag \\
    aqh-ah+a+br+b &\substack{>\\<} \frac{a^2h}{l}-\frac{ah}{l}+a+b \notag \\
    ah(a-r)-ahl+al+brl &\substack{>\\<} a^2h-ah+al \notag \\
    -arh-alh+brl &\substack{>\\<} -ah \notag \\
    b &\substack{>\\<} \frac{ah(r+l-1)}{lr} = ah 
\end{align}
Note that $b \neq \frac{ah(r+l-1)}{lr} = ah$, which would mean $b$ being a multiple of $a$, which contradicts the initial conditions. If the left-hand-side is greater than 0, it's $F'$, otherwise, it's $Q'_1$. To illustrate, we use the numbers indicated in Figure \ref{fig:h=0,a=1}. Here, $r=1$, $l=4$, $a = 13$, and $h=4$. Subbing into Inequality \ref{eq: h=0, a=1}, we have $b \substack{>\\<} 13*4 = 52$. Suppose $b = 51$, then it's a $Q'_1$. Remember that $c = -ha + lb$. Indeed $g(13,51,152) = ab - f_{[a \ b]}(Q'_1) = 444$. In this case, $Q'_1$ represents the Frobenius number. If $b = 53$, then $g(a,b,c) = g(13,53,160) = ab - f_{[a \ b]}(F') = 467$. In this case, $F'$ represents the Frobenius number. 
\subsubsection{$h \equiv 0 \pmod{l}, a \not\equiv 0, 1 \pmod{l}$}\label{subsubsec:h=0,a=2+}
We now compare to $Q'_1$ to $Q'_2$ when $a \not\equiv 0, 1 \pmod{l}$. Note that the corresponding value of $Q'_2(qh+1,1)$ is always less than $F'(\lfloor \frac{ah-h}{l}\rfloor +1) = (\frac{(ql+r)h-h}{l}+1, 1) = (qh + \frac{rh}{l}-\frac{h}{l}+1)$ for $ a \not\equiv 0, 1 \pmod{l}$. We now compare the values of $Q'_1$ and $Q'_2$:
\begin{align}\label{eq:h=0,a=2+}
f_{[a \ b]}(Q'_1) &\substack{>\\<}   f_{[a \ b]}(Q'_2) \notag \\
a((q-1)h+1) + b(r+1) &\substack{>\\<}   a(qh+1)+b(1) \notag \\
-ah+br &\substack{>\\<} 0 \notag \\
b &\substack{>\\<} \frac{ah}{r}
\end{align}
Note that if $b = \frac{ah}{r}$, $c$ is not an exceptional value. If $b > \frac{ah}{r}$, then it's a $Q'_2$ point. For example, let $r=3,l=6, a=9, h=12$. Subbing into Inequality \ref{eq:h=0,a=2+}, we have $b \substack{>\\<} 36$. Suppose $b = 37$, then $g(9,37,114) = ab - f_{[a \ b]}(Q'_2) = ab-f_{[a \ b]}(13,1)=179$. On the other hand, if $b < 36$, say 35, then $Q'_1$ is the most optimal candidate for $g(a,b,c)$: $g(9,35,102) = ab - f_{[a \ b]}(Q'_1) = ab-f_{[a \ b]}(1,4)=166$.
\subsection{Candidates and Value Comparisons for Frobenius Number for $h \not\equiv 0 \pmod{l}$}
\subsubsection{$h \not\equiv 0 \pmod{l}$, \rm{gcd}$(l,h) = 1$}\label{subsubsec: h = 1+, gcd(l,h) = 1}
Unlike before where Equation \ref{eq:key_general_sub} could be simplified into $m_0 + \frac{h}{l} = \frac{ah}{l}$ since $h \equiv 0 \pmod{l}$, we now observe what happens when $h$ is not a multiple of $l$. In the case where $g(l,h) = 1$, observe that we are now dealing with $P'$ points. For example, the most optimal lattice point might be located on the line $lm_0 + hn_0 = ah-f_{[l \ h]}(m_0,n_0)$ is not expressible as a nonnegative integer linear combination of $l$ and $h$. More specifically, $f_{[l \ h]}(m_0,n_0)$ has to be of the form $ah-(lh - xl - yh)$ so that $lm_0 + hn_0 = lh - xl -yh$ for $x,y>0$ and $lx+hy < lh$. We now specify the location of $P'$
\begin{align*}
    hl - xl -yh &= ah-f_{[l \ h]}(m_0,n_0)  \\
    -ah+lh-xl-yh &= -lm_0-hn_0 \\
    \intertext{This implies that}
    n_0 &= a-l+y \\
    m_0 &= x
\end{align*}
So, our point should be of the form $(x,a-l+y)$.  $(x,y)$ is the set of blue lattice points displayed in the figures below and we are simply increasing the values by $(a-l)h$ or $v$ coordinate by $a-l$. This is most optimal when $x,y=1$; however, like the $Q'$ case, we have suboptimal points: the other blue points in the region bounded by the line $lm_0 + hn_0 = lh$. So, all the lattice points in the top left triangular region are all considered $P$ points.
\par A lot of the properties we observed in \ref{subsec: h=0} can also be applied here. First, the aforementioned bounds where there are always solutions are the same. Second, if $(u,v)$ does not have a solution, then $(u+h,v-l)$ will also have no solution. So, for the case $g(l,h) = 1$, the locations of $P'$ and $Q'$ are the same, including the anomalies. That is, $P'_1((q-1)h+1,r+1)$ and $P'_2(qh+1,1)$. In other words, $P'$ points can be thought of as a subset of $Q'$ points where there are not nonnegative integer solutions for $c_1$ and $c_2$ in Equation \ref{eq:key_general_sub}. Therefore, the candidates and conditions are the same for $a \equiv 0, 1 \pmod{l}$ and $a \not\equiv 0,1 \pmod{l}$. Some examples are illustrated below; however, they're labeled as $P'$ points instead of $Q'$ points. 
\begin{figure}[H]
\begin{center}    
\begin{minipage}{0.33\textwidth}
\begin{tikzpicture}[scale=0.7]
\begin{axis}[
    axis lines=center,
    xlabel=$m_0$,
    ylabel=$n_0$,
    xmin=0,
    xmax=15,
    ymin=-3,
    ymax=9,
    xticklabel style={above},
    xtick={0,1, ...,9,10,11,12,13,14},
    ytick={-3,-2,-1, 0,1, ...,9,10,11,12,13,14,15},
    grid=both,
    grid style={line width=0.2pt, draw=gray!30},
    major grid style={line width=0.6pt,draw=gray!60},
    ticklabel style={font=\tiny},
    enlargelimits={abs=0.5},
    axis line style={latex-latex},
  ]
  \addplot[
    domain=-10:10,
    samples=2,
    color=red,
    thick
  ] coordinates {(0,9) (0,6)};
  \addplot[
    domain=-10:10,
    samples=2,
    color=red,
    thick
  ] coordinates {(0, 6) (5,6)};
  \addplot[
    domain=-10:10,
    samples=2,
    color=red,
    thick
  ] coordinates {(5,6) (5,3)};
  \addplot[
    domain=-10:10,
    samples=2,
    color=red,
    thick
  ] coordinates {(5,3) (10,3)};
  \addplot[
    domain=-10:10,
    samples=2,
    color=red,
    thick
  ] coordinates {(10,3) (10,0)};
  \addplot[
    domain=-10:10,
    samples=2,
    color=red,
    thick
  ] coordinates {(10,0) (15,0)};
  \addplot[
    only marks,
    mark=*,
    mark options={color=vividpurple, fill=vividpurple, scale=.5},
  ] coordinates {(11,1)};
  \node[above right, font=\normalsize] at (axis cs:11,1) {$P'_1$};
  \addplot[
    only marks,
    mark=*,
    mark options={color=vividpurple, fill=vividpurple, scale=.5},
  ] coordinates {(6,4)};
  \node[above right, font=\normalsize] at (axis cs:6,4) {$P'_{0}$};
  \addplot[
    only marks,
    mark=*,
    mark options={color=vividpurple, fill=vividpurple, scale=.5},
  ] coordinates {(1,7)};
  \node[above right, font=\normalsize] at (axis cs:1,7) {$P'_{-1}$};
  \addplot[domain=0:15, color=tan, thick]{-3/5*x+9};  
  \addplot[domain=0:15, color=vividpurple, dashed]{-3/5*x+7.4};
  \addplot[domain=0:5, color=black, thick]{-3/5*x+3};
  \addplot[domain=0:5, color=black, thick]{-3/5*x};
  \addplot[
    domain=-10:10,
    samples=2,
    color=black,
    thick
  ] coordinates {(5,0) (5,-3)};
  \filldraw[blue!20, opacity=0.5] (axis cs:0, 0) rectangle (axis cs:10,3);
  \draw[->, thick] (axis cs:10,5) -- (axis cs:10,3);
  \node[above, font=\normalsize] at (axis cs:10,5) {$hq-h$};
   \addplot[
    only marks,
    mark=*,
    mark options={color=vividpurple, fill=vividpurple, scale=.5},
  ] coordinates {(14,1)};
  \node[above left, font=\normalsize] at (axis cs:14,1) {$F'$};
\foreach \i in {2,...,4} {
    \foreach \j in {-1}{
        \addplot[only marks, mark=*, mark options={color=red, fill=red, scale=0.5}] coordinates {(\i, \j)};        
    }
}
\foreach \i in {4} {
    \foreach \j in {-2}{
        \addplot[only marks, mark=*, mark options={color=red, fill=red, scale=0.5}] coordinates {(\i, \j)};        
    }
}
\foreach \i in {2,...,4} {
    \foreach \j in {-1}{
        \addplot[only marks, mark=*, mark options={color=red, fill=red, scale=0.5}] coordinates {(\i, \j)};        
    }
}
\foreach \i in {1,...,3} {
    \foreach \j in {1}{
        \addplot[only marks, mark=*, mark options={color=blue, fill=blue, scale=0.5}] coordinates {(\i, \j)};        
    }
}
\foreach \i in {1} {
    \foreach \j in {2}{
        \addplot[only marks, mark=*, mark options={color=blue, fill=blue, scale=0.5}] coordinates {(\i, \j)};        
    }
}
\draw[->, thin] (axis cs:4,-1) -- (axis cs:1,7);
\end{axis}
\end{tikzpicture}
\caption{$h = 5\equiv 2, a =9 \equiv 0 \pmod{l=3}$}
\label{fig:h=1+,a=0}
\end{minipage}\hspace{1cm}
\begin{minipage}{0.33\textwidth}
\begin{tikzpicture}[scale=0.7]
\begin{axis}[
    axis lines=center,
    xlabel=$m_0$,
    ylabel=$n_0$,
    xmin=0,
    xmax=17,
    ymin=-3,
    ymax=10,
    xticklabel style={above},
    xtick={0,1, ...,9,10,11,12,13,14,15,16,17},
    ytick={-3,-2,-1, 0,1, ...,9,10,11,12,13,14,15},
    grid=both,
    grid style={line width=0.2pt, draw=gray!30},
    major grid style={line width=0.6pt,draw=gray!60},
    ticklabel style={font=\tiny},
    enlargelimits={abs=0.5},
    axis line style={latex-latex},
  ]
  \addplot[
    domain=-10:10,
    samples=2,
    color=red,
    thick
  ] coordinates {(0,10) (0,7)};
  \addplot[
    domain=-10:10,
    samples=2,
    color=red,
    thick
  ] coordinates {(0,7) (5,7)};
  \addplot[
    domain=-10:10,
    samples=2,
    color=red,
    thick
  ] coordinates {(5,7) (5,4)};
  \addplot[
    domain=-10:10,
    samples=2,
    color=red,
    thick
  ] coordinates {(5,4) (10,4)};
  \addplot[
    domain=-10:10,
    samples=2,
    color=red,
    thick
  ] coordinates {(10,4) (10,1)};
  \addplot[
    domain=-10:10,
    samples=2,
    color=red,
    thick
  ] coordinates {(10,1) (15,1)};
  \addplot[
    only marks,
    mark=*,
    mark options={color=vividpurple, fill=vividpurple, scale=.5},
  ] coordinates {(11,2)};
  \node[above right, font=\normalsize] at (axis cs:11,2) {$P'_1$};
  \addplot[
    only marks,
    mark=*,
    mark options={color=vividpurple, fill=vividpurple, scale=.5},
  ] coordinates {(6,5)};
  \node[above right, font=\normalsize] at (axis cs:6,5) {$P'_{0}$};
  \addplot[
    only marks,
    mark=*,
    mark options={color=vividpurple, fill=vividpurple, scale=.5},
  ] coordinates {(1,8)};
  \node[above right, font=\normalsize] at (axis cs:1,8) {$P'_{-1}$};
  \addplot[domain=0:17, color=tan, thick]{-3/5*x+10};  
  \addplot[domain=0:14.5, color=vividpurple, dashed]{-3/5*x+8.6};
  \addplot[domain=0:5, color=black, thick]{-3/5*x+3};
  \addplot[domain=0:5, color=black, thick]{-3/5*x};
  \addplot[
    domain=-10:10,
    samples=2,
    color=black,
    thick
  ] coordinates {(5,0) (5,-3)};
  \filldraw[blue!20, opacity=0.5] (axis cs:0, 1) rectangle (axis cs:10,4);
  \draw[->, thick] (axis cs:10,6) -- (axis cs:10,4);
  \node[above, font=\normalsize] at (axis cs:10,6) {$hq-h$};
  \draw[->, thick] (axis cs:15,3) -- (axis cs:15,1);
  \node[above, font=\normalsize] at (axis cs:15,3) {$qh$};
  \filldraw[green!20, opacity=0.5] (axis cs:0, 0) rectangle (axis cs:15,1);
   \addplot[
    only marks,
    mark=*,
    mark options={color=vividpurple, fill=vividpurple, scale=.5},
  ] coordinates {(16,1)};
  \node[above right, font=\normalsize] at (axis cs:16,1) {$F'$};
\foreach \i in {2,...,4} {
    \foreach \j in {-1}{
        \addplot[only marks, mark=*, mark options={color=red, fill=red, scale=0.5}] coordinates {(\i, \j)};        
    }
}
\foreach \i in {4} {
    \foreach \j in {-2}{
        \addplot[only marks, mark=*, mark options={color=red, fill=red, scale=0.5}] coordinates {(\i, \j)};        
    }
}
\foreach \i in {2,...,4} {
    \foreach \j in {-1}{
        \addplot[only marks, mark=*, mark options={color=red, fill=red, scale=0.5}] coordinates {(\i, \j)};        
    }
}
\foreach \i in {1,...,3} {
    \foreach \j in {1}{
        \addplot[only marks, mark=*, mark options={color=blue, fill=blue, scale=0.5}] coordinates {(\i, \j)};        
    }
}
\foreach \i in {1} {
    \foreach \j in {2}{
        \addplot[only marks, mark=*, mark options={color=blue, fill=blue, scale=0.5}] coordinates {(\i, \j)};        
    }
}

\draw[->, thin] (axis cs:4,-1) -- (axis cs:1,8);
\end{axis}
\end{tikzpicture}
\caption{$h = 5\equiv 2, a = 10 \equiv 1 \pmod{l=3}$}
\label{fig:h=1+,a=1}
\end{minipage}\hspace{1cm}
\begin{minipage}{0.33\textwidth}
\begin{tikzpicture}[scale=0.7]
\begin{axis}[
     axis lines=center,
    xlabel=$m_0$,
    ylabel=$n_0$,
    xmin=0,
    xmax=20,
    ymin=-3,
    ymax=11,
    xticklabel style={above},
    xtick={0,1, ...,9,10,11,12,13,14,15,16,17,18,19,20},
    ytick={-3,-2,-1, 0,1, ...,9,10,11,12,13,14,15},
    grid=both,
    grid style={line width=0.2pt, draw=gray!30},
    major grid style={line width=0.6pt,draw=gray!60},
    ticklabel style={font=\tiny},
    enlargelimits={abs=0.5},
    axis line style={latex-latex},
  ]
  \addplot[
    domain=-10:10,
    samples=2,
    color=red,
    thick
  ] coordinates {(0,11) (0,8)};
  \addplot[
    domain=-10:10,
    samples=2,
    color=red,
    thick
  ] coordinates {(0,8) (5,8)};
  \addplot[
    domain=-10:10,
    samples=2,
    color=red,
    thick
  ] coordinates {(5,8) (5,5)};
  \addplot[
    domain=-10:10,
    samples=2,
    color=red,
    thick
  ] coordinates {(5,5) (10,5)};
  \addplot[
    domain=-10:10,
    samples=2,
    color=red,
    thick
  ] coordinates {(10,5) (10,2)};
  \addplot[
    domain=-10:10,
    samples=2,
    color=red,
    thick
  ] coordinates {(10,2) (15,2)};
  \addplot[
    domain=-10:10,
    samples=2,
    color=red,
    thick
  ] coordinates {(15,2) (15,-1)};
  \addplot[
    domain=-10:10,
    samples=2,
    color=red,
    thick
  ] coordinates {(15,-1) (20,-1)};
  \addplot[
    only marks,
    mark=*,
    mark options={color=vividpurple, fill=vividpurple, scale=.5},
  ] coordinates {(11,3)};
  \node[above right, font=\normalsize] at (axis cs:11,3) {$P'_1$};
  \addplot[
    only marks,
    mark=*,
    mark options={color=vividpurple, fill=vividpurple, scale=.5},
  ] coordinates {(6,6)};
  \node[above right, font=\normalsize] at (axis cs:6,6) {$P'_{0}$};
  \addplot[
    only marks,
    mark=*,
    mark options={color=vividpurple, fill=vividpurple, scale=.5},
  ] coordinates {(1,9)};
  \node[above right, font=\normalsize] at (axis cs:1,9) {$P'_{-1}$};
  \addplot[domain=0:20, color=tan, thick]{-3/5*x+11};  
  \addplot[domain=0:17.5, color=vividpurple, dashed]{-3/5*x+9.6};
  \addplot[domain=0:5, color=black, thick]{-3/5*x+3};
  \addplot[domain=0:5, color=black, thick]{-3/5*x};
  \addplot[
    domain=-10:10,
    samples=2,
    color=black,
    thick
  ] coordinates {(5,0) (5,-3)};
  \filldraw[blue!20, opacity=0.5] (axis cs:0, 2) rectangle (axis cs:10,5);
  \draw[->, thick] (axis cs:10,7) -- (axis cs:10,5);
  \node[above, font=\normalsize] at (axis cs:10,7) {$hq-h$};
  \draw[->, thick] (axis cs:15,4) -- (axis cs:15,2);
  \node[above, font=\normalsize] at (axis cs:15,4) {$qh$};
  \filldraw[green!20, opacity=0.5] (axis cs:0, 0) rectangle (axis cs:15,2);
   \addplot[
    only marks,
    mark=*,
    mark options={color=vividpurple, fill=vividpurple, scale=.5},
  ] coordinates {(16,1)};
  \node[above right, font=\normalsize] at (axis cs:16,1) {$P'_2$};
\foreach \i in {2,...,4} {
    \foreach \j in {-1}{
        \addplot[only marks, mark=*, mark options={color=red, fill=red, scale=0.5}] coordinates {(\i, \j)};        
    }
}
\foreach \i in {4} {
    \foreach \j in {-2}{
        \addplot[only marks, mark=*, mark options={color=red, fill=red, scale=0.5}] coordinates {(\i, \j)};        
    }
}
\foreach \i in {2,...,4} {
    \foreach \j in {-1}{
        \addplot[only marks, mark=*, mark options={color=red, fill=red, scale=0.5}] coordinates {(\i, \j)};        
    }
}
\foreach \i in {1,...,3} {
    \foreach \j in {1}{
        \addplot[only marks, mark=*, mark options={color=blue, fill=blue, scale=0.5}] coordinates {(\i, \j)};        
    }
}
\foreach \i in {1} {
    \foreach \j in {2}{
        \addplot[only marks, mark=*, mark options={color=blue, fill=blue, scale=0.5}] coordinates {(\i, \j)};        
    }
}

\draw[->, thin] (axis cs:4,-1) -- (axis cs:1,9);
\end{axis}
\end{tikzpicture}
\caption{$h = 5\equiv 2, a = 11 \equiv 2 \pmod{l=3}$}
\end{minipage}%
\end{center}
  \end{figure}

\subsubsection{$h \not\equiv 0 \pmod{l}$, \rm{gcd}$(l,h) > 1$}\label{subsubsec: h=1+,gcd(l,h) > 1}
We now consider the case where $d=$\rm{gcd}$(l,h) > 1$. We observe that the process of finding $g'(a,b,c)$ is similar to the section before. The rectangular regions where there are always solutions still holds; however, if we simplify the equation $lm_0 + hn_0 = ah$ into $\frac{l}{d}m_0 + \frac{h}{d}n_0 = \frac{ah}{d} - \frac{f_{[l \ h]}(m_0,n_0)}{d}$, we're able to obtain some $P'$ points within each triangular region since $g(l,h) = d*g(l/d,h/d)$ for $d > 0$. An example of these $P'$ points within the $Q'$ region is illustrated below. 
\begin{figure}[H]
\begin{center}
\color{black}
    \begin{tikzpicture}[scale=0.7]
    \begin{axis}[
    axis lines=center,
    xmin=0,
    xmax=27,
    ymin=-2,
    ymax=16,
    xticklabel style={above},
    xtick={0,1, ...,21,22,23,24,25,26,27},
    ytick={-2,-1,0,1,...,15,16},
    grid=both,
    grid style={line width=0.2pt, draw=gray!30},
    major grid style={line width=0.6pt,draw=gray!60},
    ticklabel style={font=\tiny},
    enlargelimits={abs=0.5},
    axis line style={latex-latex},
  ]
  \foreach \i in {2} {
    \foreach \j in {-1}{
        \addplot[only marks, mark=*, mark options={color=red, fill=red, scale=0.5}] coordinates {(\i, \j)};        
    }
}
\foreach \i in {1} {
    \foreach \j in {1}{
        \addplot[only marks, mark=*, mark options={color=blue, fill=blue, scale=0.5}] coordinates {(\i, \j)};        
    }
}
  \addplot[
    domain=0:27,
    samples=100,
    color=tan,
    thick,
  ]{(16-2/3*x)};
  \addplot[
    domain=0:27,
    samples=100,
    color=violet,
    thick,
  ]{(47/3-2/3*x)};
  \addplot[
    domain=0:27,
    samples=100,
    color=vividpurple,
    thick,
  ]{(35/3-2/3*x)};
  \addplot[
    domain=0:3,
    samples=100,
    color=black,
    thick,
  ]{(2-2/3*x)};
  \addplot[
    domain=0:3,
    samples=100,
    color=black,
    thick,
  ]{(-2/3*x)};
  \addplot[
    domain=-10:10,
    samples=2,
    color=black,
    thick
  ] coordinates {(3,0) (3,-2)};
\addplot[
    domain=-10:10,
    samples=2,
    color=red,
    thick
  ] coordinates {(0,16) (0,14)};
  \addplot[
    domain=-10:10,
    samples=2,
    color=red,
    thick
  ] coordinates {(0,14) (3,14)};
  \addplot[
    domain=-10:10,
    samples=2,
    color=red,
    thick
  ] coordinates {(3,14) (3,12)};
  \addplot[
    domain=-10:10,
    samples=2,
    color=red,
    thick
  ] coordinates {(3,12) (6,12)};
  \addplot[
    domain=-10:10,
    samples=2,
    color=red,
    thick
  ] coordinates {(6,12) (6,10)};
  \addplot[
    domain=-10:10,
    samples=2,
    color=red,
    thick
  ] coordinates {(6,10) (9,10)};
  \addplot[
    domain=-10:10,
    samples=2,
    color=red,
    thick
  ] coordinates {(9,8) (9,10)};
  \addplot[
    domain=-10:10,
    samples=2,
    color=red,
    thick
  ] coordinates {(9,8) (12,8)};
  \addplot[
    domain=-10:10,
    samples=2,
    color=red,
    thick
  ] coordinates {(12,6) (12,8)};
  \addplot[
    domain=-10:10,
    samples=2,
    color=red,
    thick
  ] coordinates {(12,6) (15,6)};
  \addplot[
    domain=-10:10,
    samples=2,
    color=red,
    thick
  ] coordinates {(15,4) (15,6)};
  \addplot[
    domain=-10:10,
    samples=2,
    color=red,
    thick
  ] coordinates {(15,4) (18,4)};
  \addplot[
    domain=-10:10,
    samples=2,
    color=red,
    thick
  ] coordinates {(18,2) (18,4)};
  \addplot[
    domain=-10:10,
    samples=2,
    color=red,
    thick
  ] coordinates {(18,2) (21,2)};
  \addplot[
    domain=-10:10,
    samples=2,
    color=red,
    thick
  ] coordinates {(21,0) (21,2)};
  \addplot[
    domain=-10:10,
    samples=2,
    color=red,
    thick
  ] coordinates {(21,0) (24,0)};
  \addplot[
    domain=-10:10,
    samples=2,
    color=red,
    thick
  ] coordinates {(24,0) (24,-2)};
  \addplot[
    domain=-10:10,
    samples=2,
    color=red,
    thick
  ] coordinates {(27,-2) (24,-2)};
  \addplot[
    domain=-10:10,
    samples=2,
    color=red,
    thick
  ] coordinates {(0,16) (0,10)};
  \addplot[
    domain=-10:10,
    samples=2,
    color=red,
    thick
  ] coordinates {(9,10) (0,10)};
  \addplot[
    domain=-10:10,
    samples=2,
    color=red,
    thick
  ] coordinates {(9,10) (9,4)};
  \addplot[
    domain=-10:10,
    samples=2,
    color=red,
    thick
  ] coordinates {(18,4) (9,4)};
  \addplot[
    domain=-10:10,
    samples=2,
    color=red,
    thick
  ] coordinates {(18,4) (18,-2)};
  \addplot[
    domain=-10:10,
    samples=2,
    color=red,
    thick
  ] coordinates {(27,-2) (18,-2)};
   \addplot[
    only marks,
    mark=*,
    mark options={color=vividpurple, fill=vividpurple, scale=.5},
  ] coordinates {(1,11)};
  \node[above right, font=\normalsize] at (axis cs:1,11) {$Q'_{0}$};
   \addplot[
    only marks,
    mark=*,
    mark options={color=vividpurple, fill=vividpurple, scale=.5},
  ] coordinates {(10,5)};
  \node[above right, font=\normalsize] at (axis cs:10,5) {$Q'_{1}$};
   \addplot[
    only marks,
    mark=*,
    mark options={color=vividpurple, fill=vividpurple, scale=.5},
  ] coordinates {(19,1)};
  \node[above right, font=\normalsize] at (axis cs:19,1) {$Q'_{2}$};
  \filldraw[blue!20, opacity=0.5] (axis cs:0, 4) rectangle (axis cs:9,10);
  \draw[->, thick] (axis cs:9,14) -- (axis cs:9,10);
  \node[above, font=\normalsize] at (axis cs:9,14) {$hq-h$};
  \filldraw[green!20, opacity=0.5] (axis cs:0, 0) rectangle (axis cs:18,4);
  \draw[->, thick] (axis cs:18,8) -- (axis cs:18,4);
  \node[above, font=\normalsize] at (axis cs:18,8) {$hq$};
  \addplot[
    only marks,
    mark=*,
    mark options={color=violet, fill=violet, scale=.5},
  ] coordinates {(1,15)};
  \node[above right, font=\small] at (axis cs:1,15) {$P'_{-6}$};
  \addplot[
    only marks,
    mark=*,
    mark options={color=violet, fill=violet, scale=.5},
  ] coordinates {(4,13)};
  \node[above right, font=\small] at (axis cs:4,13) {$P'_{-5}$};
  \addplot[
    only marks,
    mark=*,
    mark options={color=violet, fill=violet, scale=.5},
  ] coordinates {(7,11)};
  \node[above right, font=\small] at (axis cs:7,11) {$P'_{-4}$};
  \addplot[
    only marks,
    mark=*,
    mark options={color=violet, fill=violet, scale=.5},
  ] coordinates {(10,9)};
  \node[above right, font=\small] at (axis cs:10,9) {$P'_{-3}$};
  \addplot[
    only marks,
    mark=*,
    mark options={color=violet, fill=violet, scale=.5},
  ] coordinates {(13,7)};
  \node[above right, font=\small] at (axis cs:13,7) {$P'_{-2}$};
  \addplot[
    only marks,
    mark=*,
    mark options={color=violet, fill=violet, scale=.5},
  ] coordinates {(16,5)};
  \node[above right, font=\small] at (axis cs:16,5) {$P'_{-1}$};
  \addplot[
    only marks,
    mark=*,
    mark options={color=violet, fill=violet, scale=.5},
  ] coordinates {(19,3)};
  \node[above right, font=\small] at (axis cs:19,3) {$P'_{0}$};
  \addplot[
    only marks,
    mark=*,
    mark options={color=violet, fill=violet, scale=.5},
  ] coordinates {(22,1)};
  \node[above right, font=\small] at (axis cs:22,1) {$P'_{1}$};
  \draw[->, thin] (axis cs:2,-1) -- (axis cs:1,15);
  \end{axis}
  \end{tikzpicture}
    \caption{$f_{[l \ h]}(m_0,n_0)=ah$ for $h=9$, $l=6$, $a=16$}
\label{subfig:Sub_Case 1.4 - 6,19,26}
\end{center}
\end{figure}
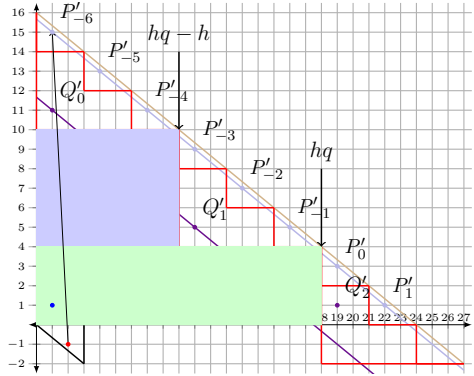
The $P'$ points within the $Q'$ triangle region, if exist, will always be above or to the right of a more optimal $Q'$ point. And thus, similar to previous discussion, these points can be neglected as possible candidates for $g'(a,b,c)$. So, when $h \not\equiv 0, g(l,h) > 1$, the possible candidates for $g'(a,b,c)$ and conditions are the same for $a \equiv 0, 1 \pmod{l}, h \equiv 0 \pmod{l}$ and $a \not\equiv 0, 1 \pmod{l}, h \equiv 0 \pmod{l}$. 
\subsubsection{Summary of $h \not\equiv 0 \pmod{l}$}
In both Section \ref{subsubsec: h = 1+, gcd(l,h) = 1} and Section \ref{subsubsec: h=1+,gcd(l,h) > 1}, note that the possible candidates for $g(a,b,c)$ are in the same locations as Section \ref{subsec: h=0}, and therefore, their value comparisons are also the same as Section \ref{subsec: value comparisons for h=0}.
\section{Final Statement}
Combining all the above discussion, we arrive at
\theorem{For relatively prime integers ${a,b,c}$ where $a=ql+r$ and $b$ are relatively prime and $c = -ha + lb$, the Frobenius number is given by the following:}\label{thm:Final}
\begin{enumerate}
  \item $a \equiv 0 \pmod{l}$
  \begin{equation}
  g(a,b,c) = ab-a((q-1)h+1)-b
  \end{equation}
    \item $a \equiv 1 \pmod{l}$
    \begin{enumerate}
    \item
    \begin{equation}
    g(a,b,c) = ab-a((q-1)h+1)-b(r+1)
    \quad \text{when } b < ah
    \end{equation}
    \item 
    \begin{equation}
    g(a,b,c) = ab-a(\lfloor \frac{ah-h}{l} \rfloor + 1)-b \quad \text{when } b > ah
    \end{equation}
    \end{enumerate}
    \item $a \not\equiv 0,1 \pmod{l}$
    \begin{enumerate}
    \item 
    \begin{equation} 
    g(a,b,c) = ab-a(qh + 1)-b \quad \text{when } b > \frac{ah}{r}
    \end{equation}
    \item
    \begin{equation}
    g(a,b,c) = ab-a((q-1)h+1)-b(r+1) \quad \text{when } b < \frac{ah}{r}
    \end{equation}
    \end{enumerate}
\end{enumerate} 
\begin{proof}
\begin{enumerate}
  \item $a \equiv 0 \pmod{l}$ \begin{align*}
    g(a,b,c) &= ab-g'(a,b,c) \\ &= ab-f_{[a \ b]}(Q'_1) \\ &= ab-f_{[a \ b]}((q-1)h+1,1) \\ &= ab-a((q-1)h+1)-b
    \end{align*}
  \item $a \equiv 1 \pmod{l}$
  \begin{enumerate}
    \item \begin{align*}
    g(a,b,c) &= ab-g'(a,b,c)\\ &= ab-f_{[a \ b]}(Q'_1)\\ &= ab-f_{[a \ b]}((q-1)h+1,r+1)\\ &= ab-a((q-1)h+1)-b(r+1) \quad \text{when } b < ah
    \end{align*}
    \item \begin{align*}
    g(a,b,c) &= ab-g'(a,b,c)\\ &= ab-f_{[a \ b]}(F')\\ &= ab-f_{[a \ b]}(\lfloor \frac{ah-h}{l} \rfloor + 1, 1)\\ &=  ab-a(\lfloor \frac{ah-h}{l} \rfloor + 1)-b(1) \quad \text{when } b > ah
    \end{align*}
  \end{enumerate}
  \item $a \not\equiv 0,1 \pmod{l}$
  \begin{enumerate}
    \item \begin{align*}
    g(a,b,c) &= ab-g'(a,b,c)\\ &= ab-f_{[a \ b]}(Q'_2)\\ &= ab-f_{[a \ b]}(qh + 1, 1)\\ &=  ab-a(qh + 1)-b(1) \quad \text{when } b > \frac{ah}{r}
    \end{align*}
    \item \begin{align*}
    g(a,b,c) &= ab-g'(a,b,c)\\ &= ab-f_{[a \ b]}(Q'_1)\\ &= ab-f_{[a \ b]}((q-1)h+1,r+1)\\ &= ab-a((q-1)h+1)-b(r+1) \quad \text{when } b < \frac{ah}{r}
    \end{align*}
  \end{enumerate}
\end{enumerate}
\end{proof}
\theorem{As an application of Theorem \ref{thm:Final}, we recover \rm{\cite{selmer1977}} explicit formula: $g(a,ha+d,ha+2d) = ab - a(K+1) - b$ where $K = h\lfloor \frac{a-1}{2} \rfloor$} and $h \geq 1, d > 0$.
\begin{proof}
Let $a=a$, $b = ha+d$, and $c =ha+2d$. Then, the lattice point $(-h,2)$ has a corresponding value of $-ha + 2(ha+d) = ha+2d = c$; therefore, $l = 2$. Since $l = 2$, $a$ is either $a \equiv 0 \pmod{l}$ or $a \equiv 1 \pmod{l}$. Note that $ b = ha+d > ah$. Therefore, $a \equiv 0 \pmod{l}$ is Case 1 in Theorem \ref{thm:Final} and $a \equiv 1 \pmod{l}$ is Case 2b in Theorem \ref{thm:Final}. Obviously, for $a \equiv 1 \pmod{l}$ for $l=2$, Case 2b is Selmer's explicit formula $g(a,ha+d,ha+2d) = ab-a(\lfloor \frac{ah-h}{l} \rfloor + 1)-b = ab-a(h\lfloor \frac{a-1}{2} \rfloor + 1)-b=ab-a(K + 1)-b$. When $a \equiv 0 \pmod{l}$ and $l=2$, $a$ is even. From Case 1, we have $g(a,b,c) = ab-a((q-1)h+1)-b = ab-a(h(\frac{a}{2}-1)+1)-b = ab-a(h\lfloor \frac{a-1}{2} \rfloor + 1) -b = ab-a(K+1)-b$. Thus, we have covered both cases, and we recover Selmer's explicit formula.
\end{proof}
\section{Discussion and Future Work}
\rm{We have developed a novel approach to solving the Frobenius problem in three variables, denoted as ${a, b, c}$, where $a$ and $b$ are relatively prime, and $c$ takes the form $-ha + lb$. This method involves the creation of a Cartesian plane, where we assign corresponding values to a set of lattice points, denoted as $D$. Lattice points located within the region $D^-$ and positioned above the line $ax + by = b$ are termed exceptional values. They result in the condition $g(a, b, c) < g(a, b)$, where $g(a, b, c)$ is in the form $ab - ax - by$ for $x, y > 0$. \par With this setup, we first explore the intriguing properties it offers before discovering two key equations (Equation \ref{eq:key_general} and Equation \ref{eq:key_general_sub}). These equations help us classify Frobenius number candidates into three categories: $F$, $P$, and $Q$ points. Once these points are identified, we observe that their corresponding values are highly dependent on the slope $\frac{a}{b}$, leading to value comparisons. The culmination of these ideas leads to the development of Theorem \ref{thm:Final}. \par The final result is straightforward compared to established works, such as \cite{tripathi2017}. An advantage of our method is its streamlined approach, with the classifications of $F$, $P$, and $Q$ points extending to cases with more than three variables, for which no existing formulas currently exist. Presently, the most efficient and well-known algorithm for calculating the Frobenius number in three variables is \cite{Rödseth1978}, which relies on continued fractions and shares the same computational complexity as the Euclidean algorithm. The only algorithmic aspect of our formula involves determining values for $c = -ha + lb$, specifically the values of $h$ and $l$. However, Theorem \ref{thm:Final} allows for easy manual computation of $g(a, b, c)$, unlike Rödseth's method.\par With regards to future work, we plan to investigate formulae for the Frobenius number in $n>3$ variables. We have outlined other potential ideas for future work; however, they're a work in progress.}
\par \vspace{12pt}
\noindent\textbf{Acknowledgements}: I would like to thank my mentor Dr. Lin Tan at West Chester University of Pennsylvania for his constant encouragement and advice, in particular for suggesting to look into the line $CZ$.
\newpage{\addcontentsline{toc}{section}{References} \bibliographystyle{plainnat}
\bibliography{references}}
\end{document}